\documentclass[reqno, a4paper, 10pt]{amsart}
\usepackage{amssymb,url, color, mathrsfs}
\usepackage[colorlinks=true, bookmarks=true, pdfstartview=FitH, pagebackref=true, linktocpage=true]{hyperref}
\usepackage[short,nodayofweek]{datetime}
\usepackage[pagewise,running,mathlines,displaymath, switch]{lineno}
\usepackage{mathabx}
\usepackage{indentfirst, tabularx}
\usepackage[table]{xcolor}

\usepackage{xcolor}

\usepackage[T5]{fontenc}
\usepackage[utf8]{inputenc}

\theoremstyle{plain}
\numberwithin{equation}{section}
\newtheorem{theorem}{Theorem}[section] 
\newtheorem{lemma}{Lemma}[section] 

\newtheorem{proposition}{Proposition}[section]

\newtheorem{remark}{Remark}[section]

\definecolor{brown}{rgb}{0.5,0,0}
\definecolor{backgroundcolor}{rgb}{0.98, 0.92, 0.73}

\def\N{\mathbb N}% natural number
\def\R{\mathbf R}
\def\C{\mathbb C}
\def\cM{\mathcal M}

\def\bB{\mathcal B} %bilinear form

\def\cn{\mathsf n}
\newif\ifprint
\ifprint
\else
\fi

\author[Olivier Lafitte]{Olivier Lafitte}
\address[O. Lafitte]{IRL 3457 CRM-CNRS, Centre de Recherche Mathematique, Universit\'e de Montr\'eal, 2940-2950 Polytechnique Rd, Montreal, Quebec H3T 1J4, Canada}
\email{\href{mailto:> lafitte@crm.umontreal.ca}{ lafitte@crm.umontreal.ca}}
\address[O. Lafitte]{Laboratoire Analyse G\'eom\'etrie et Applications, Universit\'e Sorbonne Paris Nord,  93430 - Villetaneuse, France}
\email{\href{mailto:> lafitte@math.univ-paris13.fr}{ lafitte@math.univ-paris13.fr}}

\author[Tiến-Tài Nguyễn]{ Tiến-Tài Nguyễn}
\address[T.-T. Nguyễn ]{Laboratoire Analyse G\'eom\'etrie et Applications, Universit\'e Sorbonne Paris Nord,  93430 - Villetaneuse, France}
\email{\href{mailto: T.-T. Nguyễn <tientai.nguyen@math.univ-paris13.fr>}{tientai.nguyen@math.univ-paris13.fr}}

\begin{document}
\allowdisplaybreaks

\setpagewiselinenumbers
\setlength\linenumbersep{100pt}
%\modulolinenumbers[1]
%\linenumbers

\title[Rayleigh-Taylor instability]{Spectral analysis of the incompressible viscous Rayleigh-Taylor system in $\R^3$}

\begin{abstract}

The linear instability study of the viscous Rayleigh-Taylor model in the neighborhood of a laminar smooth increasing density profile $\rho_0(x_3)$  amounts to the study of the following ordinary differential equation of order 4:
\begin{equation}\label{MainEq}
-\lambda^2 [ \rho_0 k^2 \phi - (\rho_0 \phi')'] = \lambda \mu (\phi^{(4)} - 2k^2 \phi'' + k^4 \phi) - gk^2 \rho_0'\phi,
\end{equation}
where $\lambda$ is the growth rate in time, $k$ is the wave number transverse to the density profile.

In the case of $\rho'_0\geq 0$ compactly supported, we provide a spectral analysis showing that in accordance with the results of \cite{HL03}, there is an infinite sequence of non trivial solutions $(\lambda_n, \phi_n)$ of \eqref{MainEq}, with $\lambda_n\rightarrow 0$ when $n\rightarrow +\infty$ and $\phi_n\in H^4(\R)$. In the more general case where  $\rho_0'>0$ everywhere and $\rho_0$ converges at $\pm\infty$ to finite limits $\rho_{\pm}>0$, we prove that  there exist finitely non trivial solutions $(\lambda_n, \phi_n)$ of \eqref{MainEq}. The line of investigation is to reduce  both cases to the study of a self-adjoint operator on a compact set.
\end{abstract}

\date{\bf \today \; at \, \currenttime}



\maketitle

\tableofcontents

\section*{Data availability statement}
All data generated or analysed during this study are included in this published article (and its supplementary information files).

\section{Introduction}

The Rayleigh--Taylor (RT) instability, studied first by Lord Rayleigh in \cite{Str83} and then Taylor \cite{Tay50} is well known as a gravity-driven instability in two semi-infinite inviscid and incompressible  fluids when the heavy one is on top of the light one. It  has attracted much attention due to both its physical and mathematical importance. Two applications of worth mentioning are implosion of inertial confinement fusion capsules \cite{Lin98} and core-collapse of supernovae  \cite{Rem00}.    
For a detailed physical comprehension of the RT instability, we refer to three survey papers  \cite{Kull91, Zhou17_1, Zhou17_2}. Mathematically speaking, for the inviscid and incompressible regime in the neighborhood of a smooth density profile, the classical  RT instability was investigated by Lafitte \cite{Laf01}, Guo and Hwang \cite{GH03} and  Helffer and Lafitte \cite{HL03}.

In this paper, we pay  attention to the viscous RT instability.  One of the first studies on the viscous RT can be seen in the book of Chandrasekhar \cite[Chapter X]{Cha61} considering two uniform viscous fluid separated by a horizontal boundary and  generalizing the classical result of Rayleigh and Taylor.  We are  concerned with the following Navier--Stokes equation describing  the motion of a nonhomogeneous incompressible viscous fluid in the presence of a uniform gravitational field in $\R^3$, 
\begin{equation}\label{EqNS}
\begin{cases}
\partial_t \rho +\text{div}(\rho \vec u) =0,\\
\partial_t(\rho \vec u) + \text{div}(\rho \vec u\otimes \vec u) +\nabla P =\mu \Delta\vec u - \rho \vec{g},\\
\text{div}\vec u=0,
\end{cases}
\end{equation}
where $t\geqslant 0, x=(x_1,x_2,x_3) \in \R^3$. The unknowns $\rho := \rho(x,t)$, $\vec u := \vec u(x,t)$ and $P :=P(x,t)$ denote respectively the density, the velocity and the pressure of the fluid, while $\mu>0$ is the viscosity coefficient and $ \vec{g}:= (0,0,g)$,  $g > 0$ being the gravitational constant.

Let $\rho_0>0$ be a $C^1$-function depending only on $x_3$ and let $P_0$ be another function of $x_3$ given by $P_0' = -g\rho_0 $ with $'=d/dx_3$.  Note that it is not a physical solution of the system, because $P_0\rightarrow \pm \infty$ when $x_3\rightarrow \pm \infty$, but, nevertheless, it is a relevant model to study. Then, $(\rho, \vec u,P)=(\rho_0, \vec 0,P_0)$ is an equilibrium state of \eqref{EqNS}. The linear  RT instability of  \eqref{EqNS} in the vicinity $(\rho_0(x_3),\vec 0, P_0(x_3))$  amounts to the investigation of the parameter $\lambda\in \mathbb{C}$ (Re$\lambda>0$) such that  there exists a bounded solution $\phi \in H^4(\R)$ of the following ordinary differential equation
\begin{equation}\label{4thOrderEqPhi}
-\lambda^2 ( \rho_0 k^2 \phi - (\rho_0 \phi')') = \lambda \mu (\phi^{(4)} - 2k^2 \phi'' + k^4 \phi) - gk^2 \rho_0'\phi,
\end{equation}
where $k>0$ is the transverse wave number \textit{being fixed} and $\phi$ is  the third component of the perturbed velocity. Our desired solution $\phi$ decays to zero at $\pm \infty$, i.e. 
\begin{equation}\label{BoundaryCondition}
\lim_{x_3\to \pm \infty}\phi(x_3) = 0.
\end{equation}
In this case, such a  $\lambda$ is a growth rate of the instability  or is a characteristic value of the linearized problem (see \cite[Chapter X, Sections 92-93]{Cha61}). 

The derivation of \eqref{4thOrderEqPhi} will be shown in the next section. 

Let us discuss about the mathematical study. For the inviscid problem (i.e. $\mu=0$ in \eqref{EqNS} and \eqref{4thOrderEqPhi}), \eqref{4thOrderEqPhi} can be seen as an eigenvalue problem with eigenvalue  $\lambda^2$. Multiple eigenvalues of the inviscid problem   are first mentioned by Cherfils and Lafitte \cite{CL00} through the precise study of the inviscid ODE for a specific profile. Mention also the paper of Mikaelian \cite{Mik96} for the connection between the Rayleigh \eqref{4thOrderEqPhi} and the Schrödinger equations. Lafitte \cite{Laf01} then observed that possible growth rates  $\lambda$'s for the inviscid Rayleigh-Taylor problem (not only the largest one) are such that $\lambda^2$ is an eigenvalue of a suitable self-adjoint operator and described this spectrum, for large values of $k$, as the eigenvalues of a 1D Schrodinger operator, which is generalized by Helffer and Lafitte \cite{HL03}.

On the other side, the natural variational structure of the inviscid problem pointed out in \cite{Cha61} and used by Guo and Hwang \cite{GH03} describes the square of the maximum growth rate $\lambda_1^2$ as the maximum of of the Rayleigh quotient
\[
\frac{gk^2\int_{\R}\rho_0'|\phi|^2dx_3}{\int_{\R}\rho_0(|\phi'|^2+k^2|\phi|^2)dx_3}
\]
and is used to obtain a proof of the nonlinear instability. With this argument, it can be noticed that the multiple eigenvalues are given by $\{\lambda_2,\dots, \lambda_j,\dots\}$, where $\lambda_j^2$ is equal to
\[
\max_{\phi\in H^1_k}\frac{gk^2\int_{\R}\rho_0'|\phi|^2dx_3}{\int_{\R}\rho_0(|\phi'|^2+k^2|\phi|^2)dx_3}
\]
where $H^1_k=\{\phi\in H^1(\R), \int_{\R}\phi\phi_ldx_3=0, \forall l\leq j\}$, $\phi_l$ is a nontrivial eigenfunction associated with $\lambda_l$. 
This Rayleigh quotient cannot be used here due to the presence of the viscosity term.  

For the viscous problem,  Jiang, Jiang and Ni \cite{JJN13} used   a modified variational approach, described by Guo and Tice \cite{GT11} , where a bootstrap argument yields the largest growth rate and a corresponding solution being regular under the assumption  $\rho_0' \in C_0^{\infty}(\R), \inf_{\R} \rho_0 >0$ and $\rho_0'(x_3) > 0$ for some points $x_3 \in \R$. 

As far as we could find, no other authors performed studies of the discrete spectrum of the inviscid or the viscous linearized RT instability.  We aim at the viscous study  in the mathematical spirit of Helffer-Lafitte \cite{HL03}.

We then illustrate  our spectral analysis.

Throughout this paper, we consider $\rho_0$ as an increasing function, i.e. $\rho_0'\geq 0$.  For an increasing density profile $\rho_0$, we show in Appendix \ref{Preliminaries} that $\lambda$ is real and bounded by $\sqrt{\frac{g}{L_0}}$, with $L_0^{-1}:= \sup_{\R}\frac{\rho_0'}{\rho_0}$. Since our goal is to study the unstable modes, we restrict our study to the case $\lambda >0$.

The ODE \eqref{4thOrderEqPhi} rewrites as a linear system of ODEs on $(\phi,\phi',\phi'',\phi''')^T$. As the solution of this system must tend to 0 when $x_3\rightarrow \pm \infty$, using the fact that the profile $\rho_0$ goes to $\rho_{\pm}$ at $\pm \infty$, we are able to deduce the stable linear space $S_+$ at $+ \infty$ and the unstable linear space $S_-$ at $-\infty$, that are vector spaces  of dimension two. Hence,  we transform the problem for the normal modes on $\R$ into an ODE problem stated on a compact interval $(x_-,x_+)$ with appropriate boundary conditions deduced from the outer solutions.   In the case of a compactly supported $\rho_0'$  (the simplest case of a convergence), they are described by  \eqref{LeftBoundary}-\eqref{RightBoundary}.  In the case of a non compactly supported $\rho_0'$, they are described  by \eqref{GenLeftBound}-\eqref{GenRightBound}. Note that, if $\rho_0'$ is compactly supported, the natural choice of  $x_{\pm}$ is to choose the ends of $\mbox{supp}\rho_0'$. However, if $\rho_0'$ non compactly supported, $x_{\pm}$ are deduced from the behavior at $\pm \infty$ of the outer solutions and depend also on $k$ and $\lambda$.

In order to solve \eqref{4thOrderEqPhi} on $(x_-,x_+)$, the crucial tool in our study is to construct two bilinear forms, that are continuous and coercive, $\bB_{\lambda}$ (denoted respectively by $\bB_{a,\lambda}$ \eqref{1stBilinearForm} for $\rho_0'$ being compactly supported   and  by $\bB_{x_-,x_+,\lambda}$ \eqref{2ndBilinearForm} for $\rho_0'$ being positive everwhere) such that the finding of bounded solutions of Eq. \eqref{4thOrderEqPhi} on the whole line is equivalent to the existence of solutions to  the variational problem 
\[
\lambda \bB_{\lambda}(\phi,\omega) = gk^2 \int_{x_-}^{x_+}\rho_0'\phi\omega dx_3 \quad\text{for all }\omega\in H^2((x_-,x_+)).
\]
As $\bB_{\lambda}$ is coercive on $H^2((x_-,x_+))$,  $\sqrt{\bB_{\lambda}(\cdot,\cdot)}$ is a norm on $H^2((x_-,x_+))$, and one denotes by $(H^2((x_-,x_+)))'$ the dual of $H^2((x_-,x_+))$ for the induced topology.
In view of the Riesz representation theorem, we thus define an abstract  operator 
\[
Y_{\lambda}\in \mathcal{L}(H^2((x_-,x_+)), (H^2((x_-,x_+))')
\]
(see $Y_{a,\lambda}$ in Proposition \ref{PropInverseOfR} for $\rho_0'\geq 0$  being compactly supported   and $Y_{x_-,x_+,\lambda}$ in Proposition \ref{PropPropertyB_lambda} for $\rho_0'$ being positive everywhere), associated with the norm $\sqrt{\bB_{\lambda}(\cdot,\cdot)}$ to its dual, such that 
\begin{equation}
 \bB_{\lambda}(f, \omega) = \langle Y_{\lambda}f, \omega\rangle, \quad \text{for all } \omega \in H^2((x_-,x_+))
\end{equation}
Note that it will be crucial to choose $x_-, x_+$ such that $\sqrt{\bB_{\lambda}(\cdot,\cdot)}$ is a norm on $H^2((x_-,x_+))$.
The  existence of bounded solutions of Eq. \eqref{4thOrderEqPhi} on $\R$ is thus equivalent to finding weak solutions $\phi \in H^2((x_-,x_+)$ of
\[
\lambda Y_{\lambda}\phi = gk^2\rho_0' \phi \text{ in } (H^2((x_-,x_+)))'.
\]
A classical bootstrap argument shows that $\phi$ belongs to $H^4((x_-,x_+))$, from which one deduces that $\phi$ satisfies \eqref{4thOrderEqPhi} on $(x_-,x_+)$ (in the sense of elements of $H^4((x_-,x_+))$), and the boundary conditions 
\eqref{LeftBoundary}-\eqref{RightBoundary} or \eqref{GenLeftBound}-\eqref{GenRightBound} hold. Note that, as $\phi$ belongs to $H^4((x_-,x_+))$, these boundary conditions (involving derivatives $\phi'', \phi'''$ of $\phi$ at $x_\pm$) are well defined.

We further rewrite the problem, after introducing $\cM$ the operator of multiplication by $\sqrt{\rho_0'}$ in $L^2((x_-,x_+))$, as finding $v$ (which is $\cM^{-1}Y_{\lambda}\phi$) satisfying 
\[
\cM Y_{\lambda}^{-1}\cM v =\frac{\lambda}{gk^2} v.
\]
We show that $\cM Y_{\lambda}^{-1}\cM$ is self-adjoint and compact, which enables to use the theory of self-adjoint and compact operators on the functional space $H^2((x_-,x_+))$. It is a generalization of a Sturm-Liouville problem. The discrete spectrum of the operator $\cM Y_{\lambda}^{-1}\cM$ is thus an infinite sequence of eigenvalues, denoted by $\gamma_n(\lambda)$ and let $v_n$ be an eigenfunction associated with $\gamma_n(\lambda)$.

The problem of finding characteristic values of \eqref{4thOrderEqPhi} amounts to solving all the equations
\begin{equation}\label{EqFindLambda}
\gamma_n(\lambda)= \frac{\lambda}{gk^2}.
\end{equation}

When $\rho_0'\geq 0$ is compactly supported, for each $n$,  we will show  the existence and uniqueness of a solution $\lambda_n$ to \eqref{EqFindLambda} owing first to the differentiability  in  $\lambda$ of $\gamma_n(\lambda)$ (see Lemma \ref{LemGammaCont}), which is an easy extension of  Kato's perturbation theory of \cite{Kato}, and secondly on the fact that $\lambda\rightarrow \gamma_n(\lambda)$ is decreasing in $\lambda$, through the derivative $\frac{d}{d\lambda}(\frac{1}{\gamma_n(\lambda)})$ (see Lemma \ref{LemGammaDecrease}) which exists also thanks to a similar argument of \cite{Kato}. Then $\phi_n = Y_{\lambda_n}^{-1}\cM v_n \in H^4((x_-,x_+))$ is glued with the decaying solutions of \eqref{4thOrderEqPhi} in the outer regions by the boundary conditions at $x_{\pm}$, which yields a solution of \eqref{4thOrderEqPhi} in $H^4(\R)$ associated with $\lambda_n$.  Hence, we obtain in our first result, Theorem \ref{MainThm1}, the existence of an infinite sequence of characteristic values $(\lambda_n)_{n\geq 1}$, that decreases towards 0 as $n\to \infty$. 

When $\rho_0'>0$ everywhere, unlike the first case, we lack the analytical expression of boundary conditions. We thus do not have the decrease of $\gamma_n$ or any control of $\gamma_n(\lambda)$ when $\lambda$ goes to 0.  Consequently, for $\rho_0'>0$ everywhere, our arguments only lead to a multiple existence of positive characteristic values $\lambda$ such that $\lambda\geq \epsilon_\star>0$. This is our second result, Theorem \ref{MainThm2}.

We then continue our operator method in other contexts. The first one \cite{Tai22_0} is to investigate the nonlinear Rayleigh-Taylor  instability of the gravity-driven incompressible Navier-Stokes equations with Navier-slip boundary conditions in a 2D slab domain $\Omega= 2\pi L \mathbb{T} \times (-1,1)$ with $L>0$ and $\mathbb{T}$ is the usual 1D-torus.  Based on infinitely unstable  modes of the linearized problem, we consider  a wide class of initial data to the nonlinear perturbation problem, extending Grenier's framework \cite{Gre00}, to prove  nonlinear Rayleigh-Taylor instability. 

Furthermore, as noted in \cite[Chapter X, Section 94(e)]{Cha61}, we study  the gravity waves problem that occur in an infinitely deep ocean and the occupied domain is bounded above by a moving interface $\{(t,x) \in \R_+\times 2\pi L\mathbb{T} \times \R | x_2=\eta(t,x_1)\}$.   The governing equation is  the gravity-driven incompressible Navier–Stokes equation without any  effect of surface tension  on the free surface.  The precise study will be derived in a forthcoming work \cite{Tai22_1}.

We organize this paper as follows. In Section \ref{SectResults}, we describe the physical model and then give the statement of our main theorems.  Section \ref{SectSmoothProf} and Section \ref{SectGeneralProf} are devoted to proving Theorem \ref{MainThm1} as $\rho_0$ satisfies \eqref{AssumeRho2}-\eqref{1RhoConst} and proving Theorem \ref{MainThm2} as $\rho_0$ satisfies \eqref{IntegralRho-0}-\eqref{2RhoDeri}, respectively.

\noindent{\textbf{Notation}}. Throughout this paper,  $C^{\star}$ is a generic constant depending on $\rho_0$ and other physical parameters,  independent of $\lambda$.

\section{The main results}\label{SectResults}

\subsection{Derivation of the physical model}

Recall that $(\rho, \vec u,P)=(\rho_0, \vec 0,P_0)$ is an equilibrium state of \eqref{EqNS}. The quantities 
\[\sigma =\rho-\rho_0,\quad \vec u = \vec u- \vec 0, \quad p = P-P_0\]
 satisfy the following nonlinear perturbation equations
\begin{equation}\label{EqPertur}
\begin{cases}
\partial_t \sigma + \vec u \cdot \nabla(\rho_0+\sigma)  =0,\\
(\rho_0+\sigma) \partial_t \vec u + (\rho_0+\sigma) \vec u \cdot \nabla \vec u +\nabla p =\mu \Delta \vec u - \sigma \vec g,\\
\text{div} \vec u=0.
\end{cases}
\end{equation}
That implies the following linearized system
\begin{equation}\label{EqLinearized}
\begin{cases}
\partial_t \sigma +  \rho_0' u_3=0,\\
\rho_0 \partial_t \vec u + \nabla p = \mu \Delta \vec u - \sigma \vec g,\\
\text{div}\vec u=0.
\end{cases}
\end{equation}

Since $\rho_0$ depends only on $x_3$, we continue the analysis into normal modes as in \cite[Chapter X, Section 91]{Cha61}. Precisely, we seek the perturbations of the form 
\begin{equation}\label{NormalModes}
(\sigma, \vec u, p)(t,x) = e^{\lambda t+ik_1x_1+ik_2x_2} (\zeta, -i\psi, -i\theta,\phi,q )(x_3), \end{equation}
where $k\in \R^2$, $\lambda \in \C \setminus \{0\}$ and $\text{Re} \lambda \geqslant 0$. Let $k=\sqrt{k_1^2+k_2^2}$,  we arrive at the following system 
\begin{equation}\label{SystemModes}
\begin{cases}
\lambda\zeta + \rho_0'\phi=0,\\
\lambda\rho_0 \psi- k_1 q+ \mu(k^2 \psi-\psi'')=0, \\
\lambda\rho_0\theta-k_2q +\mu(k^2 \theta-\theta'')=0,\\
\lambda\rho_0\phi + q' +\mu(k^2\phi-\phi'')+ g\zeta=0,\\
k_1\psi+k_2\theta+\phi'=0.
\end{cases}
\end{equation}
We directly see $\zeta = -\frac{\rho_0'\phi}{\lambda}$. Hence, $\eqref{SystemModes}_4$ becomes
\begin{equation}\label{SystemModes_4}
\lambda^2\rho_0 \phi+ \lambda q'+ \lambda\mu(k^2\phi-\phi')= g\rho_0'\phi. 
\end{equation}
We multiply $\eqref{SystemModes}_2$ by $k_1$ and $\eqref{SystemModes}_3$ by $k_2$, then use $\eqref{SystemModes}_4$ to obtain the equality  
\[
\lambda^2\rho_0\phi'+k^2\lambda q+\lambda \mu (k^2\phi'-\phi''')=0.
\]
Deriving this equation, and replacing $\lambda q'$ thanks to \eqref{SystemModes_4}, we get the fourth-order ordinary equation \eqref{4thOrderEqPhi}.
The investigation of multiple growing normal modes \eqref{NormalModes} amounts to finding regular solutions $\phi \in H^4(\R)$ of \eqref{4thOrderEqPhi}. These solutions decay to zero at $\pm \infty$, i.e.  $\phi$ satisfies \eqref{BoundaryCondition}.

\subsection{Main results}
Assuming the density profile $\rho_0$ is increasing, we first show that all characteristic values $\lambda$ are real and  obtain the following upper bound $\sqrt{\frac{g}{L_0}}$ of $\lambda$. The proof is postponed to Appendix \ref{Preliminaries}.
 \begin{lemma}\label{LemEigenvalueReal}
All characteristic values $\lambda$ for Eq. \eqref{4thOrderEqPhi}-\eqref{BoundaryCondition} in $H^4(\R)$ are always real and satisfy that $\lambda \leqslant \sqrt{\frac{g}{L_0}}$.
\end{lemma}

In view of Lemma \ref{LemEigenvalueReal}, we seek for functions $\phi$ being real and we only consider vector space of real functions in what follows.

In the case $\rho_0'\geq 0$ compactly supported, our assumption is  
\begin{equation}\label{AssumeRho2} 
\begin{split}
&\rho_0' \text{ is a nonnegative function  of class } C_0^0(\R), \quad \text{supp}(\rho_0') =[-a,a],  \\
\end{split}
\end{equation}
Outside $(-a,a)$, we denote
\begin{equation}\label{1RhoConst}
\quad \rho_0(x_3) = 
\begin{cases}
\rho_- &\quad\text{as } x_3\in(-\infty,-a], \\
\rho_+ &\quad\text{as } x_3 \in [a,+\infty),
\end{cases}
\end{equation}
with $\rho_- <\rho_+$ are two positive constants. This can be seen, physically speaking,  as the situation of the toy model with a layer, of size $2a$, in which there is a mixture of the two fluids of density $\rho_-$ and $\rho_+$. We have the following result in Section \ref{SectSmoothProf}.
\begin{theorem}\label{MainThm1}
Let $\rho_0$ satisfy \eqref{AssumeRho2} and \eqref{1RhoConst}. 
There exist an infinite sequence  $(\lambda_n, \phi_n)_{n\geqslant 1}$ with $\lambda_n \in (0,\sqrt{\frac{g}{L_0}})$ and  $\phi_n \in H^4(\R)$ satisfying \eqref{4thOrderEqPhi}. In addition, $\lambda_n$ decreases towards 0 as $n$ goes to $\infty$.
\end{theorem}

We divide the proof of Theorem \ref{MainThm1} into various steps. Notice that the rapid convergence is not necessary in this case because $\rho_0$ is constant on $(-\infty, -a)$ and on $(a, +\infty)$. Hence,  in this case $\phi$ can be found explicitly since \eqref{4thOrderEqPhi} outside $(-a,a)$ is an ODE with constant coefficients, that will shown in Proposition \ref{PropOdeOuterRegion}. After that, we deduce two boundary conditions \eqref{LeftBoundary} and \eqref{RightBoundary} at $\mp a$ respectively in Lemma \ref{LemBoundSmooth}. Eq. \eqref{4thOrderEqPhi} becomes a boundary-value problem on the finite interval $(-a,a)$.
Our method to solve that is based on the operator approach, namely the theory of compact and self-adjoint  operator for $Y_{a,\lambda}$, as illustrated in the introduction part.

In the next part, Section \ref{SectGeneralProf}, we consider $\rho_0'$  no longer compactly supported. The assumptions on $\rho_0$ are 
\begin{equation}\label{IntegralRho-0}
\rho_0 \in C^1(\R),  \quad \lim_{x_3 \to \pm \infty} \rho_0(x_3) = \rho_{\pm} \in (0,+\infty)
\end{equation}
and 
\begin{equation}\label{2RhoDeri}
0<\rho_0'(x_3) < \rho_m <+\infty \text{ for all } x_3 \in \R.
\end{equation}
The second theorem  is as follows.
\begin{theorem}\label{MainThm2}
Let $\rho_0$ satisfy  \eqref{IntegralRho-0} and \eqref{2RhoDeri}. 
For  $0<\epsilon_\star\ll 1$, there exists $N(\epsilon_\star)\in \N^\star$ such that there are at least $N(\epsilon_\star)$ values of $\lambda\in [\epsilon_\star, \sqrt{\frac{g}{L_0}}]$ such that $\phi\in H^4(\R)$  satisfying \eqref{4thOrderEqPhi}.
\end{theorem}

The proof of Theorem \ref{MainThm2} remains the same to that one of the first case, but more complicated. We point out  the main differences as follows.

Questions concerning the existence of solutions of \eqref{4thOrderEqPhi} being bounded at $\infty$ are not trivial as in the first case. In Section \ref{SectGeneralProf}, we transform Eq. \eqref{4thOrderEqPhi} into a system of ODEs \eqref{EqDiffU}. The matrix $L(x_3,\lambda)$ has 4 eigenvalues $\pm k$ and $\pm \sqrt{k^2+\lambda\rho_0(x_3)/\mu}$, that are different for all $\lambda>0$. We then follow \cite[Theorem 8.1, Chapter 3]{CL81}, whose statement given in Appendix \ref{ThmGeneralSystem}, to  deduce that \eqref{4thOrderEqPhi}  admits two linearly independent solutions decaying to 0 at $+\infty$ (respectively $-\infty$). A suitable  interval $(x_-, x_+)$ is thus determined through a  precise calculation of the family of solutions decaying exponentially to zero at $\pm \infty$, which yield appropriate boundary conditions \eqref{GenLeftBound} at $x_-$ and \eqref{GenRightBound} at $x_+$ in Proposition \ref{PropEquivalentEq}.

We then solve \eqref{4thOrderEqPhi} on a finite interval $(x_-,x_+)$ with boundary conditions  \eqref{GenLeftBound}--\eqref{GenRightBound}. To do that, in Section \ref{SubSolInterval},  we  construct the bilinear form $\bB_{x_-,x_+,\lambda}$  in Proposition \ref{PropInverseGeneralT} and continue the same arguments as in Section \ref{SectSmoothProf} to obtain solution in the inner region $(x_-,x_+)$.  Note that the coercivity of $\bB_{x_-,x_+,\lambda}$ relies on the positivity of the terms $BV_{x_+,\lambda}$ and $BV_{x_-,\lambda}$ stated in Lemma \ref{CoercivityLemma}. Due to the lack of analytical expression of boundary conditions in this case, it turns out that  the positivity of $BV_{x_+,\lambda}$ and $BV_{x_-,\lambda}$ will be derived, in Proposition \ref{PropInverseGeneralT}, by  deducing the behavior at $\pm \infty$  of coefficients $n_{ij}^{\pm}$ ($i,j=1,2$) depending on $(x_\pm,\lambda)$ and appearing in the boundary conditions \eqref{GenLeftBound}-\eqref{GenRightBound}.  Having the bilinear form $\bB_{x_-,x_+,\lambda}$, we continue our arguments in Propositions \ref{PropPropertyB_lambda} and \ref{PropFiniteLambda}, that follows the same line  of Section \ref{SectSmoothProf}, to prove Theorem \ref{MainThm2}.

Note that, Lemma \ref{CoercivityLemma} does not give any control on $x_-$ and $x_+$. It is interesting for computational purposes as well as for a study of particular profiles (for example profiles decaying exponentially to their limit at $\pm \infty$), to be able to derive an explicit interval on which this is true. Notice that the restriction $\lambda\geq \epsilon_\star>0$ of Theorem \ref{MainThm2} implies that the matrix $R(\lambda)$ (see \eqref{MatrixRemainder}) in Eq. \eqref{EqDiffU} becomes regular. Hence, for profile $\rho_0'>0$ everywhere, we devote Section \ref{RefinedEstimates} to establish a control of $x_-$ and $x_+$ independent of $\lambda$. In Propositions \ref{PropSolUplusInfty} and \ref{PropSolUminusInfty}, through a careful construction of Volterra series, we can obtain refined estimates of bounded solutions of \eqref{EqDiffU} at $\infty$ uniformly in $\lambda\in [\epsilon_\star, \sqrt{\frac{g}{L_0}}]$. They allow us to have refined estimates on the coefficients $n_{ij}^{\pm}$ $(i,j=1,2)$ uniformly in $\lambda\in [\epsilon_\star, \sqrt{\frac{g}{L_0}}]$. Hence, we obtain a criterion for $x_-$ and $x_+$  in Proposition \ref{coercive-interval} to fulfill the conditions of Lemma \ref{CoercivityLemma} and extend Proposition \ref{PropInverseGeneralT}.

\section{The compactly supported profile}\label{SectSmoothProf}

In this section, we consider $\rho_0$ satisfying \eqref{AssumeRho2} and \eqref{1RhoConst}.   We remark that in this section, we use the notations $\nu_{\pm}= \frac{\rho_{\pm}}{\mu}$ and  $\tau_{\pm}(\lambda)=(k^2+\lambda\nu_{\pm})^{1/2}$ and  throughout this paper, we will use $x$ instead of $x_3$ for notational conveniences.

\subsection{The solution in outer regions and reduction to a problem on a finite interval}\label{SolOuterRegion}

We derive, in this subsection, the precise expression of $\phi(x)$ as  $|x|\geqslant a$.  
\begin{proposition}\label{PropOdeOuterRegion} 
There are two linearly independent solutions of \eqref{4thOrderEqPhi}  decaying to 0 at $+\infty$ as $x\in [a,+\infty)$, i.e.
\begin{equation}\label{BaseRightSol}
\phi_1^+(x)= e^{-kx}\quad\text{and}\quad\phi_2^+(x)=e^{-\tau_+(\lambda) x}.
\end{equation}
and two linearly independent solutions of \eqref{4thOrderEqPhi} decaying to 0 at $-\infty$ as $x \in (-\infty,-a]$, i.e. 
\begin{equation}\label{BaseLeftSol}
\phi_1^-(x)= e^{kx}\quad\text{and}\quad\phi_2^-(x)=e^{\tau_-(\lambda) x}.
\end{equation}
All solutions decaying to 0 at $+ \infty$ (respectively at $-\infty$) are spanned by $(\phi_1^+,\phi_2^+)$ (respectively by $(\phi_1^-,\phi_2^-)$).
\end{proposition}
\begin{proof}
For  $x \in[a,+\infty)$,  \eqref{4thOrderEqPhi} reduces to 
\begin{equation}
-\lambda \nu_{+} ( k^2 \phi - \phi'') = \phi^{(4)} - 2k^2 \phi'' +k^4 \phi.
\end{equation}
We seek $\phi$ as $\phi(x) = e^{rx}$. Hence, 
\[
-\lambda \nu_+ (k^2-r^2) = r^4-2k^2r^2+k^4,
\]
which yields $r= \pm k$ or $r=\pm (k^2+\lambda \nu_+)^{1/2}$. Since $\phi$ tends to 0 at $+\infty$,  we get two linearly independent solutions
\[
\phi_1^+(x)= e^{-kx}\quad\text{and}\quad\phi_2^+(x)=e^{-\tau_+(\lambda)x}.
\]
Hence, all solutions $\phi$ decaying to 0 at $+\infty$ are of the form 
\begin{equation}\label{LeftSol}
\phi(x) = A_1^+ e^{-k(x-a)}+ A_2^+ e^{-\tau_+(\lambda)(x-a)}
\end{equation}
for all $x\in [a,+\infty)$ and for some real constants $A_1^+$ and $A_2^+$. 

If $x \in (-\infty,-a]$, the same calculation implies \eqref{BaseRightSol}. Then, all solutions $\phi$ decaying to 0 at $-\infty$ are of the form 
\begin{equation}\label{RightSol}
\phi(x) = A_1^- e^{k(x+a)}+A_2^- e^{\tau_-(\lambda)(x+a)}
\end{equation}
for all $x\in (-\infty,-a]$ and for some real constants $A_1^-$ and $A_2^-$.
\end{proof}

Once it is proven that $\phi(x)$ outside $(-a,a)$ is of the form \eqref{LeftSol} and \eqref{RightSol}, we search for $\phi$ on $(-a,a)$. That solution has to match with \eqref{LeftSol} and \eqref{RightSol} well, i.e. there is a condition  on $(\phi,\phi',\phi'',\phi''')$ at $x=\pm a$. We will show the conditions in the following lemma. 

\begin{lemma}\label{LemBoundSmooth}
The boundary condition of \eqref{4thOrderEqPhi} at $x=-a$, for $\phi\in H^4(\R)$, is 
\begin{equation}\label{LeftBoundary}
\begin{cases}
k\tau_{-} \phi(-a) -(k+\tau_-(\lambda)) \phi'(-a) +\phi''(-a)=0,\\
k\tau_-(\lambda)(k+\tau_-(\lambda)) \phi(-a) - (k^2+k\tau_-(\lambda)+\tau_-^2(\lambda))\phi'(-a)+\phi'''(-a)=0.
\end{cases}
\end{equation}
and at $x=a$ is 
\begin{equation}\label{RightBoundary}
\begin{cases}
k\tau_+(\lambda) \phi(a) +(k+\tau_+(\lambda)) \phi'(a) +\phi''(a)=0,\\
-k\tau_+(\lambda)(k+\tau_+(\lambda)) \phi(a) - (k^2+k\tau_+(\lambda) +\tau_+^2(\lambda)) \phi'(a)+\phi'''(a)=0.
\end{cases}
\end{equation}
\end{lemma}
\begin{proof}
The boundary condition of the solutions $\phi$ of \eqref{4thOrderEqPhi} at $x= \pm a$ is equivalent to the fact that $\phi$ belongs to the space of decaying solutions at $\pm \infty$. On the one hand, it can be seen from \eqref{LeftSol} and \eqref{RightSol} that 
\[
\begin{pmatrix}
\phi(x) \\ \phi'(x) \\ \phi''(x) \\ \phi'''(x) 
\end{pmatrix} 
= A_1^- e^{k(x+a)}
\begin{pmatrix}
1 \\ k \\ k^2 \\k^3
\end{pmatrix}
+A_2^- e^{\tau_-(\lambda)(x+a)}
\begin{pmatrix}
1 \\ \tau_-(\lambda) \\ \tau_-^2(\lambda) \\ \tau_-^3(\lambda)
\end{pmatrix}
\quad\text{for }x\leqslant -a
\]
and that
\[
\begin{pmatrix}
\phi(x) \\ \phi'(x) \\ \phi''(x) \\ \phi'''(x) 
\end{pmatrix} 
=A_1^+ e^{-k(x-a)}
\begin{pmatrix}
1 \\ -k \\ k^2 \\ -k^3
\end{pmatrix}
+A_2^+ e^{-\tau_+(\lambda)(x-a)}
\begin{pmatrix}
1 \\ -\tau_+(\lambda) \\ \tau_+^2(\lambda) \\ -\tau_+^3(\lambda)
\end{pmatrix},
\quad\text{for }x\geqslant a.
\]
On the other hand,  the orthogonal complement of the subspace $S_-$ of $\R^4$ spanned by two vectors $(1,k, k^2,k^3)^T$ and $(1,\tau_-(\lambda),\tau_-^2(\lambda),\tau_-^3(\lambda))^T$ is spanned by 
\[
(k\tau_-(\lambda),-(k+\tau_-(\lambda)),1,0)^T \text{ and }(k\tau_-(\lambda)(k+\tau_-(\lambda)),-(k^2+k\tau_-(\lambda)+\tau_-(\lambda)^2),0,1)^T.
\]  
Similarly, the orthogonal complement of the subspace $S_+$ of $\R^4$ spanned by  two vectors $(1,-k, k^2,-k^3)^T$ and $(1,-\tau_+(\lambda),\tau_+^2(\lambda),-\tau_+^3(\lambda))$  is spanned by two vectors 
\[
(k\tau_+(\lambda),k+\tau_+(\lambda),1,0)^T \quad\text{and} \quad(-k\tau_+(\lambda)(k+\tau_+(\lambda)), -(k^2+k\tau_+(\lambda) +\tau_+^2(\lambda)),0,1)^T.
\]  
The above arguments allow us to set \eqref{LeftBoundary} and \eqref{RightBoundary} as boundary conditions of Eq. \eqref{4thOrderEqPhi} on $(-a,a)$.

Remark that $S_-$ is also spanned by $\{(1,k, k^2,k^3)^T, (0, 1, k+\tau_-(\lambda), k^2+k\tau_-(\lambda)+\tau_-^2(\lambda))^T\}$, and that $S_+$ is also spanned by $\{(1,-k, k^2,-k^3)^T, (0, 1, k-\tau_+(\lambda), k^2-k\tau_+(\lambda)+\tau_+^2(\lambda))^T\}$. This choice of generating families yields two uniformly independent families when $\lambda\rightarrow 0$.
\end{proof}

We aim at solving \eqref{4thOrderEqPhi} on $(-a,a)$ with the boundary conditions \eqref{LeftBoundary}-\eqref{RightBoundary}.

\subsection{A bilinear form and a self-adjoint invertible  operator}\label{SubSectSmoothOpe}

We introduce the bilinear form $\bB_{a,\lambda}$ in the following proposition. 

\begin{proposition}\label{PropPropertyR}
Let us denote by 
\begin{equation}\label{ValueThetaRhoAt-A}
\begin{split}
BV_{-a,\lambda}(\vartheta,\varrho) &:= 
\mu \left( \begin{split}
&k\tau_-(\lambda)(k+\tau_-(\lambda)) \vartheta(-a) \varrho(-a) - k\tau_-(\lambda)\vartheta'(-a) \varrho(-a) \\
&\qquad -  k\tau_-(\lambda) \vartheta(-a) \varrho'(-a) + (k+\tau_-(\lambda))\vartheta'(-a) \varrho'(-a)
\end{split} \right)
\end{split}
\end{equation}
and by
\begin{equation}\label{ValueThetaRhoAtA}
\begin{split}
BV_{a,\lambda}(\vartheta,\varrho)
&:=\mu \left( \begin{split}
&k\tau_+(\lambda)(k+\tau_+(\lambda)) \vartheta(a) \varrho(a) - k\tau_+(\lambda)\vartheta'(a) \varrho(a) \\
&\qquad -  k\tau_+(\lambda) \vartheta(a) \varrho'(a) + (k+\tau_+(\lambda))\vartheta'(a) \varrho'(a)
\end{split} \right).
\end{split}
\end{equation}
Then, 
\begin{equation}\label{1stBilinearForm}
\begin{split}
\bB_{a,\lambda}(\vartheta, \varrho) & := BV_{a,\lambda}(\vartheta,\varrho) + BV_{-a,\lambda}(\vartheta,\varrho) + \lambda \int_{-a}^a  \rho_0 (k^2\vartheta  \varrho + \vartheta'  \varrho') dx \\
&\qquad\qquad+ \mu \int_{-a}^a (\vartheta''  \varrho'' + 2k^2 \vartheta'  \varrho' +k^4 \vartheta  \varrho)dx.
\end{split} 
\end{equation}
is a continuous and coercive bilinear form on $H^2((-a,a))$.

Furthermore, let  $(H^2((-a,a)))'$ be the dual space of $H^2((-a,a))$ associated with the norm $\sqrt{\bB_{a,\lambda}(\cdot,\cdot)}$, there exists a unique operator  
\[Y_{a,\lambda} \in  \mathcal{L}(H^2((-a,a)), (H^2((-a,a)))'),\] that is also bijective,  such that
 \begin{equation}\label{EqMathcalB_a}
\bB_{a,\lambda}(\vartheta, \varrho) = \langle Y_{a,\lambda}\vartheta,  \varrho\rangle
\end{equation}
for all $\vartheta, \varrho \in H^2((-a,a))$.
\end{proposition}
\begin{proof}
Clearly, $\bB_{a,\lambda}$ is a bilinear form on  $H^2((-a,a))$ since the terms  $BV_{\pm a,\lambda}(\vartheta,\varrho)$ are well defined. We then establish the boundedness of $\bB_{a,\lambda}$.  The integral terms of $\bB_{a,\lambda}$ are bounded by 
\begin{equation}\label{Bound34termB}
\max\Big( \rho_+ k^2\sqrt{\frac{g}{L_0}} + \mu k^4,\rho_+\sqrt{\frac{g}{L_0}} +2\mu k^2, \mu\Big) \|\vartheta\|_{H^2((-a,a))}\|\varrho\|_{H^2((-a,a))}. 
\end{equation}
About the two first terms $BV_{\pm a,\lambda}(\vartheta,\varrho)$, it follows from the general Sobolev inequality that 
\[
\| \vartheta(y)\|_{C^{0,j}((-a,a))} \leqslant C(j,a) \|\vartheta(y)\|_{H^1((-a,a))} \quad\text{for all } j \in[0,1/2). 
\]
Therefore, we obtain 
\[
|\vartheta(a)|^2+ |\vartheta(-a)|^2 \leqslant C^{\star} \|\vartheta\|_{H^1((-a,a))}^2.
\]
Similarly, 
\[
|\vartheta'(a)|^2+ |\vartheta'(-a)|^2 \leqslant C^{\star} \|\vartheta'\|_{H^1((-a,a))}^2.
\]
Consequently, we get 
\begin{equation}\label{Bound12termB}
\begin{split}
|BV_{\pm a,\lambda}(\vartheta,\varrho)| &\leqslant C^{\star}  ( |\vartheta(\pm a)|+\vartheta'(\pm a)|) ( |\varrho(\pm a)|+ |\varrho'(\pm a)| ) \\
&\leqslant C^{\star} \|\vartheta\|_{H^2((-a,a))} \|\varrho\|_{H^2((-a,a))}.
\end{split}
\end{equation}
In view of \eqref{Bound34termB} and \eqref{Bound12termB}, we find that 
\begin{equation}\label{BoundBcontinuous}
|\bB_{a,\lambda}(\vartheta, \varrho)| \leqslant C^{\star}  \|\vartheta\|_{H^2((-a,a))} \|\varrho\|_{H^2((-a,a))},
\end{equation}
i.e. $\bB_{a,\lambda}$ is bounded.

We move to show the coercivity of $\bB_{a,\lambda}$. We have that 
\[
\begin{split}
\bB_{a,\lambda}(\vartheta,\vartheta) &= BV_{a,\lambda}(\vartheta,\vartheta) + BV_{-a,\lambda}(\vartheta,\vartheta) + \lambda \int_{-a}^a  \rho_0 (k^2|\vartheta|^2  + |\vartheta'|^2 ) dx \\
&\qquad+  \mu \int_{-a}^a (|\vartheta''|^2 + 2k^2 |\vartheta'|^2 +k^4 |\vartheta|^2 )dx.
\end{split}
\]
$BV_{a,\lambda}(\vartheta,\vartheta) \geqslant 0$ follows from the following equality 
\[
\begin{split}
\frac1{\mu}BV_{a,\lambda}(\vartheta,\vartheta) &=   k\tau_+(\lambda)(k+\tau_+(\lambda)) |\vartheta(a)|^2 +2 k\tau_+(\lambda) \vartheta(a)\vartheta'(a) \\
&\qquad+ (k+\tau_+(\lambda))|\vartheta'(a)|^2\\
&= k\tau_+(\lambda)(k+\tau_+(\lambda)) \Big| \vartheta(a) + \frac{\vartheta'(a)}{k+\tau_+(\lambda)}\Big|^2 \\
&\qquad+ \frac{k^2+k\tau_+(\lambda)+\tau_+^2(\lambda)}{k+\tau_+(\lambda)} |\vartheta'(a)|^2.
\end{split}
\]
We also obtain that $BV_{-a, \lambda}(\vartheta,\vartheta) \geqslant 0$. Therefore, we deduce that 
\begin{equation}\label{LowerBoundB}
\bB_{a,\lambda}(\vartheta,\vartheta) \geqslant  \mu \min (k^4,  2k^2, 1) \|\vartheta\|_{H^2((-a,a))}^2.
\end{equation}
It then tells us that $\bB_{a,\lambda}$ is a  continuous and coercive bilinear form on $H^2((-a,a))$. It follows from Reisz's representation theorem  that there is a unique operator $Y_{a,\lambda} \in  \mathcal{L}(H^2((-a,a)), (H^2((-a,a)))')$, that is also bijective, satisfying \eqref{EqMathcalB_a} for all $\vartheta, \varrho \in H^2((-a,a))$. Proof of Proposition \ref{PropPropertyR} is complete.
\end{proof}

The next proposition is to devoted to studying the properties of $Y_{a,\lambda}$.

\begin{proposition}\label{PropInverseOfR}
We have the following results.
\begin{enumerate}
\item For all $\vartheta \in H^2((-a,a))$, 
\[
Y_{a,\lambda}\vartheta=\lambda (\rho_0k^2\vartheta-(\rho_0\vartheta')')+\mu(\vartheta^{(4)}-2k^2 \vartheta''+k^4 \vartheta)
\]  in $ \mathcal{D}'((-a,a))$.  

\item Let $f\in L^2((-a,a))$ be given, there exists a unique solution  $\vartheta \in H^2((-a,a))$ of \begin{equation}\label{EqY=f}
Y_{a,\lambda}\vartheta = f \text{ in } ( H^2((-a,a)))',
\end{equation}
then $\vartheta\in H^4((-a,a))$ and satisfies the boundary conditions \eqref{LeftBoundary}--\eqref{RightBoundary}.
\end{enumerate}
\end{proposition}

\begin{proof}
Let  $\varrho \in C_0^{\infty}((-a,a))$, it follows from Proposition \ref{PropPropertyR} that, for $\vartheta\in H^2(-a,a)$ 
\begin{equation}\label{EqIntegralBform}
\lambda \int_{-a}^a  \rho_0 (k^2\vartheta  \varrho + \vartheta'  \varrho') dx + \mu \int_{-a}^a (\vartheta''  \varrho'' + 2k^2 \vartheta'  \varrho' +k^4 \vartheta  \varrho)dx = \langle Y_{a,\lambda}\vartheta,  \varrho \rangle.
\end{equation}
We respectively define $(\vartheta'')'$ and $(\vartheta'')''$ in the distributional sense as the first and second derivative of $\vartheta''$ which is in $L^2((-a,a))$. Hence, \eqref{EqIntegralBform} is equivalent to
\begin{equation}\label{EqIntegralBform2}
\lambda \int_{-a}^a  \rho_0 (k^2\vartheta  \varrho + \vartheta'  \varrho') dx + \mu \int_{-a}^a( 2k^2 \vartheta'  \varrho' +k^4 \vartheta  \varrho)dx+ \langle (\vartheta'')'' ,\varrho\rangle = \langle Y_{a,\lambda}\vartheta,  \varrho \rangle.
\end{equation}
for all $\varrho \in C_0^{\infty}((-a,a))$.  We deduce from \eqref{EqIntegralBform2} that
\begin{equation}\label{EqIntegralTransformB}
 \int_{-a}^a (\lambda(k^2 \rho_0 \vartheta -(\rho_0 \vartheta')')+\mu( -2k^2 \vartheta'' + k^4 \vartheta))  \varrho dx + \mu \langle (\vartheta'')'' ,  \varrho \rangle = \langle Y_{a,\lambda}\vartheta,  \varrho \rangle
\end{equation}
for all $\varrho \in C_0^{\infty}((-a,a))$. The first assertion follows.

Let $f\in L^2((-a,a))$ and $\vartheta \in H^2((-a,a))$ be the solution of \eqref{EqY=f},  we then improve the regularity of the weak solution $\vartheta$ of \eqref{EqIntegralTransformB}. Indeed,  we rewrite \eqref{EqIntegralTransformB} as 
\begin{equation}\label{2ndEqIntegralTransformB}
\mu \langle (\vartheta'')'' , \varrho\rangle =  \int_{-a}^a (f+2\mu k^2 \vartheta'' -\mu k^4 \vartheta - \lambda k^2 \rho_0 \vartheta + \lambda (\rho_0 \vartheta')')  \varrho dx
\end{equation}
for all $\varrho \in C_0^{\infty}((-a,a))$. Since $(f+2\mu k^2 \vartheta'' -\mu k^4 \vartheta - \lambda k^2 \rho_0 \vartheta + \lambda (\rho_0 \vartheta')')$ belongs to $L^2((-a,a))$, it then follows from \eqref{2ndEqIntegralTransformB} that   $(\vartheta'')'' \in L^2((-a,a))$.  Furthermore, by usual distribution theory, we define $\Psi \in \mathcal{D}'((-a,a))$ such that 
\begin{equation}\label{EqDefinePsi}
\langle \Psi, \varrho \rangle = \langle (\vartheta'')'', \zeta_{\rho} \rangle
\end{equation}
for all $\varrho \in C_0^{\infty}$, where $\zeta_{\rho}(x) = \int_{-a}^x( \varrho(y) - \int_{-a}^a \varrho(s) ds) dy$. Hence, it can be seen that 
\[
\langle \Psi', \varrho \rangle = - \langle \Psi, \varrho' \rangle = -  \langle (\vartheta'')'', \zeta_{\rho'} \rangle = -\langle (\vartheta'')'',  \varrho \rangle
\]
that implies $(\vartheta'')'- \Psi \equiv \text{constant}$. In view of $(\vartheta'')'' \in L^2((-a,a))$ and \eqref{EqDefinePsi}, we know that $(\vartheta'')' \in L^2((-a,a))$. Since $\vartheta \in H^2((-a,a))$ and $(\vartheta'')', (\vartheta'')'' \in L^2((-a,a))$, it tells us that $\vartheta$ belongs to $H^4((-a,a))$.

By exploiting \eqref{2ndEqIntegralTransformB}, we then show that $\vartheta$ satisfies \eqref{LeftBoundary} and \eqref{RightBoundary}. Indeed, consider now $\varrho \in C^{\infty}((-a,a))$, one has, using the integration by parts 
\[
\begin{split}
\int_{-a}^a(\vartheta'')''(x) \varrho(x)dx&=(\vartheta'')'(a) \varrho(a)-(\vartheta'')'(-a) \varrho(-a)-(\vartheta'')(a) \varrho'(a)\\
&\qquad+(\vartheta'')(-a) \varrho'(-a)+\int_{-a}^a\vartheta''(x) \varrho''(x)dx.
\end{split}
\]
We perform on the other terms of \eqref{EqIntegralTransformB} the integration by parts, which yields
\[
\begin{split}
 &\lambda \int_{-a}^a  \rho_0 (k^2\vartheta  \varrho + \vartheta'  \varrho') dx +  \mu \int_{-a}^a (\vartheta''  \varrho'' + 2k^2 \vartheta'  \varrho' +k^4 \vartheta  \varrho)dx \\
&\quad- \lambda \rho_0\vartheta' \varrho \Big|_{-a}^a  + \mu \Big(  \vartheta''' \varrho \Big|_{-a}^a - \vartheta'' \varrho' \Big|_{-a}^a - 2k^2  \vartheta' \varrho \Big|_{-a}^a \Big) = \int_{-a}^a (Y_{a,\lambda}\vartheta ){\varrho} dx.
\end{split}
\]
It then follows from the definition of the bilinear form $\bB_{a,\lambda}$ that
\begin{equation}\label{EqBvImply}
BV_{a,\lambda}(\vartheta,\varrho) + BV_{-a,\lambda}(\vartheta,\varrho) = - \lambda \rho_0\vartheta' \varrho \Big|_{-a}^a  + \mu \Big(  \vartheta''' \varrho \Big|_{-a}^a - \vartheta'' \varrho' \Big|_{-a}^a - 2k^2  \vartheta' \varrho \Big|_{-a}^a \Big)
\end{equation}
for all  $\varrho \in C^{\infty}((-a,a))$.

By collecting all terms  corresponding to $\varrho(-a)$ in  \eqref{EqBvImply}, we deduce that
\[
\begin{split}
&\mu(k\tau_-(\lambda)(k+\tau_-(\lambda)) \vartheta(-a) -k\tau_-(\lambda) \vartheta'(-a)) \\
&\qquad=\lambda \rho_0(-a) \vartheta'(-a) - \mu( \vartheta'''(-a) - 2k^2 \vartheta'(-a)). 
\end{split}
\]
It yields
\[
\vartheta''(-a) - (k^2+k\tau_-(\lambda)+\tau_-^2(\lambda)) \vartheta'(-a)+ k\tau_-(\lambda)(k+\tau_-(\lambda)) \vartheta(-a) =0
\]
owing to the definition of $\tau_-(\lambda)$. Then, we collect all terms  corresponding to $\varrho(a)$ or  to $ \varrho'(\pm a)$ in  \eqref{EqBvImply} to conclude that $\vartheta$ satisfies \eqref{LeftBoundary} and \eqref{RightBoundary}.
This ends the proof of Proposition \ref{PropPropertyR}.
\end{proof}

We have the following proposition on $Y_{a,\lambda}^{-1}$.
\begin{proposition}\label{RemNormR}
The operator $Y_{a,\lambda}^{-1} : L^2((-a,a)) \to L^2((-a,a))$ is compact and self-adjoint. 
\end{proposition}
\begin{proof}
It follows from Proposition \ref{PropInverseOfR} that $Y_{a,\lambda}$  admits an inverse operator $Y_{a,\lambda}^{-1}$ from $L^2((-a,a))$ to a subspace of $H^4((-a,a))$ requiring all elements satisfy \eqref{LeftBoundary}--\eqref{RightBoundary}, which is  symmetric  due to Proposition \ref{PropPropertyR}. We compose $Y_{a,\lambda}^{-1}$ with the continuous injection from $H^4((-a,a))$ to $L^2((-a,a))$. Notice that  the embedding $H^p((-a,a)) \hookrightarrow H^q((-a,a))$ for $p>q\geqslant 0$ is compact. Therefore,   $Y_{a,\lambda}^{-1}$ is compact and self-adjoint from $L^2((-a,a))$ to $L^2((-a,a))$.  Proposition \ref{RemNormR} is shown.
\end{proof}

\begin{remark}
In this paper, we choose to define the operator 
\[
\phi \mapsto  \lambda( \rho_0 k^2 \phi - (\rho_0 \phi')') +\mu (\phi^{(4)} - 2k^2 \phi'' + k^4 \phi) = Y_{a,\lambda}\phi
\]
with  boundary conditions \eqref{LeftBoundary}-\eqref{RightBoundary} through the Riesz representation theorem. We can also  define that by the following way.

$Y_{a,\lambda}$ is well defined on 
\[
D(Y_{a,\lambda})=\{ \phi \in C^4((-a,a)), \phi \text{ verifies } \eqref{LeftBoundary}-\eqref{RightBoundary} \}
\]
 and that we can extend $Y_{a,\lambda}$ over the closure of $D(Y_{a,\lambda}$).  Furthermore,  $Y_{a,\lambda}$ with the domain $H^4((-a,a))$ containing functions that satisfy \eqref{LeftBoundary}-\eqref{RightBoundary} is symmetric and positive. It follows from Friedrichs extension (see \cite[Theorem  4.3.1]{Hel10} e.g.) that $Y_{a,\lambda}$ admits a self-adjoint extension. 
\end{remark}

\subsection{A sequence of characteristic values}\label{BasisSolution}

We continue considering $\lambda \in (0,\sqrt{\frac{g}{L_0}}]$ and study the operator $S_{a,\lambda} := \cM Y_{a,\lambda}^{-1}\cM$, where $\cM$ is the operator of multiplication by $\sqrt{\rho'_0}$. Note that this choice prevents to consider a case where $\rho_0'$ could be negative.

\begin{proposition}
The operator $S_{a,\lambda} : L^2((-a,a)) \to L^2((-a,a))$ is compact and self-adjoint, under the hypothesis \eqref{AssumeRho2}. 
\end{proposition}
\begin{proof}
Due to the boundedness of $\rho_0'$, the operator $S_{a,\lambda}$ is well-defined from $L^2((-a,a))$ to itself. $Y_{a,\lambda}^{-1}$ is compact, so is $S_{a,\lambda}$.   Moreover, because both  the inverse $Y_{a,\lambda}^{-1}$ and $\cM$ are self-adjoint, the self-adjointness of $S_{a,\lambda}$ follows. 
\end{proof}
As a result of the spectral theory of compact and self-adjoint operators, the point spectrum of $S_{a,\lambda}$ is discrete, i.e. is a decreasing sequence $\{\gamma_n(k,\lambda)\}_{n\geqslant 1}$ of  positive eigenvalues of $S_{a,\lambda}$ that tends to 0 as $n\to \infty$, associated with normalized orthogonal eigenvectors $\{\varpi_n\}_{n\geqslant 1}$ in $L^2((-a,a))$. That means 
\[
\gamma_n(\lambda)\varpi_n = \cM Y_{a,\lambda}^{-1}\cM \varpi_n.
\]
so that with $\phi_n = Y_{a,\lambda}^{-1}\cM \varpi_n \in H^4((-a,a))$, one has
\begin{equation}\label{EqRf_n}
\gamma_n(\lambda) Y_{a,\lambda}\phi_n =  \rho_0' \phi_n
\end{equation}
and $\phi_n$ satisfies \eqref{LeftBoundary}-\eqref{RightBoundary}.
\eqref{EqRf_n} also tells us that $\gamma_n(\lambda) >0$ for all $n$. Indeed, we obtain 
\[
\gamma_n(\lambda) \int_{-a}^a (Y_{a,\lambda}\phi_n)  \phi_n dx = \int_{-a}^a \rho_0'|\phi_n|^2 dx.
\]
That implies
\begin{equation}\label{EqPhi_nB}
\gamma_n(\lambda) \bB_{a,\lambda}(\phi_n,\phi_n) = \int_{-a}^a \rho_0'|\phi_n|^2 dx.
\end{equation}
Since $\bB_{a,\lambda}(\phi_n,\phi_n) >0$ and $\rho_0' >0$ on $(-a,a)$, we know that $\gamma_n(\lambda)$ is positive.

For each $n$,  in order to verify that $\phi_n$ is a solution of \eqref{4thOrderEqPhi}, we are left to look for real values of $\lambda_n$  such that \eqref{EqFindLambda}.
To solve \eqref{EqFindLambda}, we have to prove that $\gamma_n(\lambda)$ is differentiable and decreasing in terms of $\lambda$, respectively in two next lemmas.

\begin{lemma}\label{LemGammaCont}
For each $n$, the functions $\gamma_n(\lambda)$ and $\phi_n$ are differentiable in terms of $\lambda \in (0,\sqrt{\frac{g}{L_0}}]$.
\end{lemma}
\begin{proof}

The family $(Y_{a,\lambda})_{\lambda \in (0,\sqrt{\frac{g}{L_0}}]}$ is a family of  bounded  operators owing to Proposition \ref{RemNormR}. 
It can be seen that the boundary conditions \eqref{LeftBoundary}-\eqref{RightBoundary}  differentiable in the parameter $\lambda$ will tell us that $(Y_{a,\lambda})_{\lambda \in (0,\sqrt{\frac{g}{L_0}}]}$ is also a family of differentiable operators  on $\lambda$ by following  a generalized treatment of \cite[Example 1.15, Chapter VII, \$1.6]{Kato}.
Since $Y_{a, \lambda}^{-1}$ exists for all $\lambda \in (0,\sqrt{\frac{g}{L_0}}]$, it follows from \cite[Theorem 2.23, Chapter IV, \$2.6]{Kato} that $Y_{a,\lambda}^{-1}$ is differentiable for all $\lambda \in (0,\sqrt{\frac{g}{L_0}}]$, so is $S_{a,\lambda}$. We then apply the differentiable property of eigenvalues of self-adjoint and compact operators, demonstrated in Appendix \ref{AppContinuous} to deduce that $\gamma_n(\lambda)$ and $\phi_n$ are differentiable functions.
\end{proof}

\begin{lemma}\label{LemGammaDecrease}
For each $n$, the function $\gamma_n(\lambda)$ is decreasing in $\lambda \in (0,\sqrt{\frac{g}{L_0}}]$.
\end{lemma}
\begin{proof}
Let $z_n= \frac{d\phi_n}{d\lambda}$. Note that $\frac{d}{d\lambda}Y_{a,\lambda}u=\rho_0k^2u-(\rho_0u')'$.  It follows from \eqref{EqRf_n} that 
\begin{equation}\label{1stEqDeriTz_n}
\begin{split}
k^2\rho_0\phi_n- (\rho_0\phi_n')'+ Y_{a,\lambda}z_n = \frac1{\gamma_n(\lambda)} \rho_0'z_n+ \frac{d}{d\lambda}\Big( \frac1{\gamma_n(\lambda)}\Big)\rho_0'\phi_n
\end{split}
\end{equation}
on $(-a,a)$. In addition, we have that at $x=-a$,
\begin{equation}\label{ZnAtMinusA}
\begin{cases}
z_n''(-a)-(k+\tau_-(\lambda))z_n'(-a)+k\tau_-(\lambda) z_n(-a)\\
\qquad\qquad= \frac{\nu_-}{2\tau_-(\lambda)} \phi_n'(-a) - \frac{k\nu_-}{2\tau_-(\lambda)}\phi_n(-a),\\
z_n'''(-a)- (k^2+k\tau_-(\lambda)+\tau_-^2(\lambda))z_n'(-a)+k\tau_-(\lambda)(k+\tau_-(\lambda))z_n(-a) \\
\qquad\qquad= \Big(\frac{k\nu_-}{2\tau_-(\lambda)}+\nu_-\Big)\phi_n'(-a) - \Big(\frac{k^2\nu_-}{2\tau_-(\lambda)}+k\nu_-\Big)\phi_n(-a) 
\end{cases}
\end{equation}
and that at $x=a$, 
\begin{equation}\label{ZnAtPlusA}
\begin{cases}
z_n''(a)+(k+\tau_+(\lambda))z_n'(a)+k\tau_+(\lambda) z_n(a) \\
\qquad\qquad= - \frac{\nu_+}{2\tau_+(\lambda)} \phi_n'(a) - \frac{k\nu_+}{2\tau_+(\lambda)}\phi_n(a),\\
z_n'''(a)- (k^2+k\tau_+(\lambda)+\tau_+^2(\lambda))z_n'(a)- k\tau_+(\lambda)(k+\tau_+(\lambda))z_n(a) \\
\qquad\qquad= \Big(\frac{k\nu_+}{2\tau_+(\lambda)}+\nu_+\Big)\phi_n'(a) + \Big(\frac{k^2\nu_+}{2\tau_+(\lambda)}+k\nu_+\Big)\phi_n(a).
\end{cases}
\end{equation}
Multiplying by $\psi_n$ on both sides of \eqref{1stEqDeriTz_n} and then integrating by parts to obtain that 
\begin{equation}\label{2ndEqDeriTz_n}
\begin{split}
&\int_{-a}^a (k^2\rho_0\phi_n- (\rho_0\phi_n')') \phi_n dx+ \int_{-a}^a (Y_{a,\lambda}z_n) \phi_n dx \\
&\qquad\qquad= \frac1{\gamma_n(\lambda)} \int_{-a}^a \rho_0'z_n \phi_n dx+ \frac{d}{d\lambda}\Big( \frac1{\gamma_n(\lambda)}\Big) \int_{-a}^a\rho_0'|\phi_n|^2 dx.
\end{split}
\end{equation}
Thanks to the integration by parts, we have 
\begin{equation}\label{2ndEqDeriTz_nPsi_n}
\int_{-a}^a (k^2\rho_0\phi_n- (\rho_0\phi_n')') \phi_n dx = \int_{-a}^a \rho_0(k^2|\phi_n|^2+|\phi_n'|^2)dx - (\rho_0\phi_n' \phi_n)\Big|_{-a}^a
\end{equation}
and 
\begin{equation}\label{2ndEqDeriTz_nTz_n}
\begin{split}
\int_{-a}^a (Y_{a,\lambda}z_n) \phi_n dx&=  \int_{-a}^a (Y_{a,\lambda} \phi_n)z_n dx + \Big( \mu(z_n'''\phi_n -z_n''\phi_n' -2k^2z_n'  \phi_n)-\lambda\rho_0 z_n' \phi_n \Big)\Big|_{-a}^a \\
&\qquad- \Big(\mu(\phi_n'''z_n -\phi_n'' z_n'-2k^2\phi_n' z_n)-\lambda\rho_0 \phi_n'z_n\Big)\Big|_{-a}^a.
\end{split}
\end{equation}
Owing to \eqref{2ndEqDeriTz_nPsi_n}, \eqref{2ndEqDeriTz_nTz_n} and \eqref{EqRf_n}, \eqref{2ndEqDeriTz_n} becomes 
\begin{equation}\label{3rdEqDeriTz_n}
\begin{split}
&\int_{-a}^a \rho_0(k^2|\phi_n|^2+|\phi_n'|^2)dx   + \Big( \mu(z_n'''\phi_n -z_n''\phi_n' -2k^2z_n'  \phi_n)-\lambda\rho_0 z_n' \phi_n \Big)\Big|_{-a}^a \\
&\qquad\quad\qquad- \Big(\mu(\phi_n'''z_n -\phi_n'' z_n'-2k^2\phi_n' z_n)-\lambda\rho_0 \phi_n'z_n\Big)\Big|_{-a}^a- (\rho_0\phi_n' \phi_n)\Big|_{-a}^a\\
&=\frac{d}{d\lambda}\Big( \frac1{\gamma_n(\lambda)}\Big) \int_{-a}^a\rho_0'|\phi_n|^2 dx.
\end{split}
\end{equation}
Using \eqref{ZnAtMinusA}, we obtain 
\[
\begin{split}
&-\Big(\mu(z_n'''\phi_n -z_n''\phi_n' -2k^2z_n'  \phi_n)-\lambda\rho_0 z_n' \phi_n\Big)(-a) \\
&\qquad+ \Big(\mu(\phi_n'''z_n -\phi_n'' z_n'-2k^2\phi_n' z_n)-\lambda\rho_0 \phi_n'z_n\Big)(-a) +\rho_-\phi_n'(-a)\phi_n(-a)\\
&= \mu \Big( \frac{k^2\nu_-}{2\tau_-(\lambda)}+k\nu_-\Big)|\phi_n(-a)|^2-\mu \Big( \frac{k\nu_-}{2\tau_-(\lambda)}+\nu_-\Big)\phi_n'(-a)\phi_n(-a) \\
&\qquad- \mu \frac{k\nu_-}{2\tau_-(\lambda)}\phi_n(-a)\phi_n'(-a) +\mu\frac{\nu_-}{2\tau_-(\lambda)}|\phi_n'(-a)|^2 + \rho_-\phi_n'(-a)\phi_n(-a).
\end{split}
\]
Keep in mind that $\mu\nu_- =\rho_-$, one has 
\begin{equation}\label{4thEqDeriTz_n}
\begin{split}
&-\Big(\mu(z_n'''\phi_n -z_n''\phi_n' -2k^2z_n'  \phi_n)-\lambda\rho_0 z_n' \phi_n\Big)(-a) \\
&\qquad+ \Big(\mu(\phi_n'''z_n -\phi_n'' z_n'-2k^2\phi_n' z_n)-\lambda\rho_0 \phi_n'z_n\Big)(-a) +\rho_-\phi_n'(-a)\phi_n(-a)\\
&=k\rho_-|\phi_n(-a)|^2 + \frac{\rho_-}{2\tau_-(\lambda)}|\phi_n'(-a)-k\phi_n(-a)|^2.
\end{split}
\end{equation}
Similarly, using \eqref{ZnAtPlusA}, we obtain 
\begin{equation}\label{5thEqDeriTz_n}
\begin{split}
&\Big(\mu(z_n'''\phi_n -z_n''\phi_n' -2k^2z_n'  \phi_n)-\lambda\rho_0 z_n' \phi_n\Big)(a) \\
&\qquad- \Big(\mu(\phi_n'''z_n -\phi_n'' z_n'-2k^2\phi_n' z_n)-\lambda\rho_0 \phi_n'z_n\Big)(a) - \rho_-\phi_n'(a)\phi_n(a)\\
&= k\rho_+|\phi_n(a)|^2 + \frac{\rho_+}{2\tau_+(\lambda)}|\phi_n'(a)-k\phi_n(a)|^2.
\end{split}
\end{equation}
Combining \eqref{3rdEqDeriTz_n}, \eqref{4thEqDeriTz_n} and \eqref{5thEqDeriTz_n}, we deduce that
\begin{equation}
\begin{split}
&\frac{d}{d\lambda}\Big(\frac1{\gamma_n(\lambda)}\Big)\int_{-a}^a\rho_0'|\phi_n|^2dx\\&=\int_{-a}^a \rho_0(k^2|\phi_n|^2+|\phi_n'|^2)dx  + k\rho_-|\phi_n(-a)|^2 + \frac{\rho_-}{2\tau_-(\lambda)}|\phi_n'(-a)-k\phi_n(-a)|^2\\
&\qquad\quad+ k\rho_+|\phi_n(a)|^2 + \frac{\rho_+}{2\tau_+(\lambda)}|\phi_n'(a)-k\phi_n(a)|^2.
\end{split}
\end{equation}
It yields that $\frac1{\gamma_n(\lambda)}$ is increasing in $\lambda$, i.e. $\gamma_n(\lambda)$ is decreasing in $\lambda$. This ends the proof of Lemma \ref{LemGammaDecrease}.
\end{proof}
Now we are in proposition to solve \eqref{EqFindLambda}.
\begin{proposition}\label{PropExistLambda}
For each $n$, there exists only one positive $\lambda_n$ satisfying \eqref{EqFindLambda}.  In addition, $\lambda_n$ decreases towards 0 as $n$ goes to $\infty$.
\end{proposition}
\begin{proof}
Using \eqref{EqPhi_nB}, we know that
\[
\frac1{\gamma_n(\lambda)} \int_{-a}^a \rho_0' |\phi_n|^2 dx = \int_{-a}^a (Y_{a,\lambda}\phi_n) \phi_n dx = \bB_{a,\lambda}(\phi_n,\phi_n), 
\]
that implies
\[
\frac1{L_0\gamma_n(\lambda)}  \geqslant \lambda  k^2+ \frac{\mu k^4}{\rho_+}.
\]
Consequently, for all $n\geqslant 1$,
\begin{equation}\label{LimitGammaRight}
\lim_{\lambda \to \sqrt{\frac{g}{L_0}}} \frac{\lambda}{\gamma_n(\lambda)} > gk^2.
\end{equation}
As $0<\frac{1}{\gamma_n(\lambda)}$ and it is a increasing function, $\frac{1}{\gamma_n(\lambda)}\leq \frac{1}{\gamma_n(\frac12\sqrt{\frac{g}{L_0}})}$ for all $\lambda\leq \frac12 \sqrt{\frac{g}{L_0}}$. That implies
\begin{equation}\label{LimitGammaLeft}
\lim_{\lambda \to 0} \frac{\lambda}{\gamma_n(\lambda)} \leq \lim_{\lambda\to 0} \frac{\lambda}{\gamma_n(\frac12\sqrt\frac{g}{L_0})}= 0 \quad\text{for all } n\geqslant 1.
\end{equation}
Combining \eqref{LimitGammaRight}, \eqref{LimitGammaLeft} and using Lemma \ref{LemGammaDecrease}, we deduce that there is only one $\lambda_n \in (0,\sqrt{\frac{g}{L_0}})$ satisfying \eqref{EqFindLambda} for each $n\geqslant 1$.

We then prove that the sequence $(\lambda_n)_{n\geq 1}$ is  decreasing. Indeed, if $\lambda_m<\lambda_{m+1}$ for some $m\geq 1$, we have  
\[
\gamma_m(\lambda_m) > \gamma_m(\lambda_{m+1}).
\]
Meanwhile, we also have
\[
\gamma_m(\lambda_{m+1}) > \gamma_{m+1}(\lambda_{m+1}).
\]
That implies 
\[
\frac{\lambda_m}{gk^2}= \gamma_m(\lambda_m) >  \gamma_{m+1}(\lambda_{m+1}) = \frac{\lambda_{m+1}}{gk^2}.
\]
That contradiction tells us that $(\lambda_n)_{n\geq 1}$ is a decreasing sequence.  Suppose that 
\[
\lim_{n\to \infty} \lambda_n = c_0 > 0.
\]
Note that 
\[
\gamma_n(c_0) \geq \gamma_n(\lambda_n) = \frac{\lambda_n}{gk^2}.
\]
Let $n\to \infty$, we get a contradiction that $0\geq c_0$. Hence, $\lambda_n$ decreases towards 0 as $n\to \infty$.
We conclude Proposition \ref{PropExistLambda}.
\end{proof}

\subsection{Proof of Theorem \ref{MainThm1}}

Let $\lambda_n$ be found from Proposition \ref{PropExistLambda} and $\phi_n(x)= Y_{a,\lambda_n}^{-1}\cM \varpi_n(x)$ on $(-a,a)$. 
Keep in mind our computations in Subsection \ref{SolOuterRegion}, we extend $\phi_n$ to the whole line by requiring $\phi_n$ satisfies \eqref{LeftSol} and \eqref{RightSol}
for some constants $A_{n,1}^{\pm}$ and  $A_{n,2}^{\pm}$ as $\lambda=\lambda_n$. Those constants $A_{n,1}^{\pm}$ and  $A_{n,2}^{\pm}$ are defined by 
\begin{equation}\label{EqA+}
\begin{cases}
\phi_n(a) = A_{n,1}^+ + A_{n,2}^+, \\
\phi_n'(a) = kA_{n,1}^+ +  \sqrt{k^2+\lambda_n \nu_+}A_{n,2}^+.
\end{cases}
\end{equation}
and by
\begin{equation}\label{EqA-}
\begin{cases}
\phi_n(-a) = A_{n,1}^- + A_{n,2}^-, \\
\phi_n'(-a) = kA_{n,1}^- +   \sqrt{k^2+\lambda_n \nu_-}A_{n,2}^-.
\end{cases}
\end{equation}
Since $\phi_n $ is not trivial, there exists  $i\geqslant 0$ such that $\phi_n^{(i)}(a) \neq 0$, which ensures that, at least, one of $A_{n,1}^{\pm}$  is nonzero. Similarly,  at least, one of $A_{n,2}^{\pm}$ is nonzero.

Solving \eqref{EqA+} and \eqref{EqA-}, we get that 
\begin{equation}\label{ConstA+}
A_{n,1}^+ = \frac{\sqrt{k^2+\lambda_n \nu_+} \phi_n(a) - \phi_n'(a)}{\sqrt{k^2+\lambda_n \nu_+} -k}, \quad A_{n,2}^+ = \frac{\phi_n'(a) -k \phi_n(a) }{\sqrt{k^2+\lambda_n \nu_+} -k}.
\end{equation}
and that
\begin{equation}\label{ConstA-}
A_{n,1}^- = \frac{\sqrt{k^2+\lambda_n \nu_-} \phi_n(-a) - \phi_n'(-a)}{\sqrt{k^2+\lambda_n \nu_-} -k}, \quad A_{n,2}^- = \frac{\phi_n'(-a) -k \phi_n(-a) }{\sqrt{k^2+\lambda_n \nu_-} -k}.
\end{equation}
Therefore, the function $\phi_n$ is a regular solution of \eqref{4thOrderEqPhi} as $\lambda =\lambda_n$ for each $n\geqslant 1$. Proof of Theorem \ref{MainThm1} is complete.

\section{The strictly increasing profile case}\label{SectGeneralProf}

In this section,  let $(e_1,e_2,e_3,e_4)$ be the canonical basis of $\R^4$ and we pay  attention to an increasing profile $\rho_0$ satisfying  \eqref{IntegralRho-0} and \eqref{2RhoDeri}.  We will use the Frobenius matrix norm $\|\cdot\|_F$  and the Euclidean vector norm $\|\cdot\|_2$.

\subsection{Solutions decaying to 0 at infinity and reduction to a problem on a finite interval}

 Let
\begin{equation}\label{MatrixL}
L(x,\lambda) = \begin{pmatrix}
0 & 1 &0 &0 \\
0 & 0 & 1 & 0\\
0 & 0 & 0 & 1\\
-\frac{\lambda k^2 \rho_0(x)}{\mu} - k^4 & 0 & \frac{\lambda \rho_0(x)}{\mu}+2k^2 & 0 
\end{pmatrix}
\end{equation}
and 
\begin{equation}\label{MatrixRemainder}
R(\lambda) =  \begin{pmatrix}
0 & 0 & 0 & 0 \\
0 & 0 & 0 & 0\\
0 & 0 & 0 & 0\\
\frac{gk^2}{\lambda \mu} & \frac{\lambda}{\mu} & 0 &0
\end{pmatrix}
\end{equation}
We set $U = (\phi,\phi',\phi'',\phi''')^T$ and then rewrite \eqref{4thOrderEqPhi} as 
\begin{equation}\label{EqDiffU}
U'(x) =  (L(x,\lambda)+ \rho_0'(x)R(\lambda))U(x), 
\end{equation}
The eigenvalues of $L(x,\lambda)$, $\pm k$ and $\pm \sigma_0(x,\lambda)$, with $\sigma_0(x,\lambda)= \sqrt{k^2 + \frac{\lambda }{\mu}\rho_0(x)}$ are different for all $\lambda>0$ and for all $x\in \R$.  Furthermore,
\begin{equation}\label{ConditionL_R}
 \int_{-\infty}^{+ \infty}\vert L'(x,\lambda)\vert dx<+\infty, \quad \int_{-\infty}^{+\infty}\vert \rho_0'(x)R(\lambda)\vert dx<+\infty.
\end{equation}
\eqref{ConditionL_R} allow us to use \cite[Theorem 8.1, Chapter 3]{CL81} (see Appendix \ref{ThmGeneralSystem} for the statement)  to find bounded solutions of \eqref{EqDiffU} near $\pm \infty$.
We denote  $\sigma_\pm(\lambda)= \lim_{x\to \pm\infty}\sigma_0(x,\lambda)$.  Then, there exist $\bar x_-<\bar x_+$ such that we have the existence of bounded solutions $U_{1,2}^+$ on $(\bar x_+, +\infty)$ and $U_{3,4}^-$ on $(-\infty,\bar x_-)$ such that in a  neighborhood of $+\infty$, 
\begin{equation}\label{defiU+}\begin{array}{l}
e^{kx} U_1^+(x, \lambda) \rightarrow (-k^{-3},k^{-2}, -k^{-1}, 1)^T ,\\
\text{exp}\Big( \int_{{\bar x}_+}^x \sigma_0(y,\lambda)dy\Big)U_2^+(x, \lambda)\rightarrow (-\sigma_+^{-3}(\lambda), \sigma_+^{-2}(\lambda), -\sigma_+^{-1}(\lambda), 1)^T, \end{array}\end{equation}
and in a neighborhood of $-\infty$,
\begin{equation}\label{defiU-}\begin{array}{l}
e^{-kx } U_3^-(x, \lambda) \rightarrow (k^{-3},k^{-2}, k^{-1}, 1)^T, \\
\text{exp}\Big( -\int_{{\bar x}_-}^x \sigma_0(y,\lambda)dy\Big)U_4^-(x, \lambda) \rightarrow(\sigma_-^{-3}(\lambda), \sigma_-^{-2}(\lambda), \sigma_-^{-1}(\lambda), 1)^T.\end{array}\end{equation}

Let us prove that $U_1^+(x,\lambda)$ and $U_2^+(x, \lambda)$ are linearly independent. Through Cauchy-Lipschitz theorem, if they are linearly dependent at a particular $x_0$ then they are linearly dependent for all $x$, that is there exists $T(\lambda)$ such that 
\[
U_1^+(x, \lambda)=T(\lambda)U_2^+(x, \lambda).
\]
 In particular, as the limit of $kx-\int_{{\bar x}_+}^x \sigma_0(y,\lambda)dy$ when $x\rightarrow +\infty$ is equal to $-\infty$, one observes that 
\[
e^{kx}U_1^{+}(x, \lambda)= T(\lambda)\text{exp}\Big( kx-\int_{{\bar x}_+}^x \sigma_0(y,\lambda)dy\Big) \Big(U_2^+(x,\lambda) \text{exp}\Big(\int_{{\bar x}_+}^x \sigma_0(y,\lambda)dy\Big)\Big).
\] 
The right hand side of this identity converges to 0 when $x\to +\infty$, while the left hand side converges to $(-k^{-3},k^{-2}, -k^{-1}, 1)^T$ when $x\to +\infty$, contradiction. Similarly, $U_3^-(x, \lambda)$ and $U_4^-(x, \lambda)$ are linearly independent.

The aim of the next proposition is to reduce the study of Eq. \eqref{4thOrderEqPhi} on the real line to its study on a finite interval as in the previous section. This is really the key of our result for a smooth general profile because it entitles us to use the compact injection of $H^{j}((a,b))$ into $H^{j-1}((a, b)), j\geq 1$.

\begin{proposition}\label{PropEquivalentEq}
There exist $x_-^0\leq \bar x_-$ and $x_+^0\geq \bar x_+$ such that, for all $x_-\leq x_-^0$ and $x_+\geq  x_+^0$, there are constants $n_{ij}^{\pm}$ $ (i,j=1,2)$, depending on $x_{\pm}$ and $\lambda$ such that equation 
\[
\lambda^2(k^2\rho_0\phi-(\rho_0\phi')') +\lambda\mu(\phi^{(4)}-2k^2\phi''+k^4\phi) =gk^2\rho_0'\phi,
\]
 where $\phi \in H^4((x_-,x_+))$, supplemented with the boundary conditions at $x_-$ are \begin{equation}\label{GenLeftBound}
\begin{cases}
n_{11}^{-} \phi(x_-) + n_{12}^{-} \phi'(x_-) +  \phi''(x_{-})  = 0,\\
n_{21}^{-} \phi(x_-) + n_{22}^{-} \phi'(x_{-})  +  \phi'''(x_{-}) = 0
\end{cases}
\end{equation}
and at $x_+$, are
\begin{equation}\label{GenRightBound}
\begin{cases}
n_{11}^{+} \phi(x_+) + n_{12}^{+} \phi'(x_+) +  \phi''(x_{+})  = 0,\\
n_{21}^{+} \phi(x_{+}) + n_{22}^{+} \phi'(x_{+}) + \phi'''(x_{+}) = 0,
\end{cases}
\end{equation}
is equivalent to equation
\[
\lambda^2(k^2\rho_0\Phi-(\rho_0\Phi')') +\lambda\mu(\Phi^{(4)}-2k^2\Phi''+k^4\Phi) =gk^2\rho_0'\Phi,
\]
 where $\Phi \in  H^4(\R)$. 
 
In this case $\Phi=\phi$ on $(x_-,x_+)$ and on $(x_+,+\infty)$ (respectively $(-\infty,x_-)$), $\Phi$ is the first component  of  a linear combination of $U_{1,2}^+(x,\lambda)$ (see \eqref{defiU+}) (respectively $U_{3,4}^-(x,\lambda)$, see \eqref{defiU-}).
\end{proposition}
\begin{proof}
Notice that  if $\Phi$ is a bounded solution of 
\[
\lambda^2(k^2\rho_0\Phi-(\rho_0\Phi')') +\lambda\mu(\Phi^{(4)}-2k^2\Phi''+k^4\Phi) =gk^2\rho_0'\Phi,
\]
$\Theta = (\Phi, \Phi', \Phi'', \Phi''')^T$ is a bounded solution of \eqref{EqDiffU}. 
Any decaying solution $\Psi$ of \eqref{EqDiffU} on $(\bar x_+, +\infty)$ belongs to the space spanned by $U_1^+(x, \lambda)$ and  $U_2^+(x, \lambda)$, which is of dimension 2 because they are linearly independent. It is equivalent to say that, for any $x_+\geq {\bar x}_+$, 
\begin{equation}
U_1^+(x_+, \lambda)\wedge U_2^+(x_+, \lambda)\wedge \Theta(x_+, \lambda)=0\label{eq-manifold+}
\end{equation}
\centerline{and}
\centerline{$\Theta(x_+, \lambda)$ belongs to the space spanned by $U_1^+(x_+, \lambda)$ and $U_2^+(x_+, \lambda)$}. 
\vskip 0.3 cm
Let us write $U_i^+ = (U_{i1}^+, U_{i2}^+, U_{i3}^+, U_{i4}^+)^T$ for $i=1,2$ and $U_i^-=(U_{i1}^-, U_{i2}^-, U_{i3}^-, U_{i4}^-)^T$ for $i=3,4$.
System (\ref{eq-manifold+}) is a system of four equations on the components of $\Theta(x_+,\lambda)$, hence there exists a couple of equations which are linearly independent (the system being of rank 2). Let us notice that two of these four equations contain, in $\Phi''$ and $\Phi'''$, respectively the term 
\[
(U_{11}^+U_{22}^+-U_{12}^+U_{21}^+)(x_+,\lambda)\Phi''(x_+, \lambda)
\]
 and 
 \[
 (U_{11}^+U_{22}^+-U_{12}^+U_{21}^+)(x_+,\lambda)\Phi'''(x_+, \lambda).
 \]
 As the limit, when $x_+\rightarrow +\infty$, of 
\[
\text{exp}\Big( kx_+ + \int_{\bar x_+}^{x_+} \sigma_0(y,\lambda) dy\Big)(U_{11}^+U_{22}^+-U_{12}^+U_{21}^+)(x_+,\lambda)
\]
 is 
\[
-\frac1{k^3\sigma_+^2(\lambda)}+\frac1{k^2\sigma_+^3(\lambda)} = -\frac{\lambda \rho_+}{\mu k^3\sigma_+^3(\lambda) (k+\sigma_+)}<0,
\]
 by continuity there exists a $x_+^0\geq \bar x_+$ such that,  for all $x_+\geq x_+^0$, 
\[
 (U_{11}^+U_{22}^+-U_{12}^+U_{21}^+)(x_+,\lambda)<0
\] hence the equations for (\ref{eq-manifold+}) on the components $e_1\wedge e_2\wedge e_3$ and  $e_1\wedge e_2\wedge e_4$ write as $N^+ \Theta(x_+,\lambda)=0$ with $N^+$ is a $4\times 2$ matrix of the form 
\[
N^+ = \begin{pmatrix} n_{11}^+ & n_{12}^+ & 1 &0 \\
n_{21}^+ & n_{22}^+ & 0 & 1
\end{pmatrix}.
\]
We  are now able to write the couple $N^+ \Theta(x_+,\lambda)=0$ as 
\[
\begin{cases}
n_{11}^{+} \Phi(x_+,\lambda) + n_{12}^{+} \Phi'(x_+,\lambda) +  \Phi''(x_{+},\lambda)  = 0,\\
n_{21}^{+} \Phi(x_{+},\lambda) + n_{22}^{+} \Phi'(x_{+},\lambda) + \Phi'''(x_{+},\lambda) = 0,
\end{cases}
\]
In a similar way, there exist $n_{ij}^-$ $(i,j=1,2)$ depending on $x_-$ and $\lambda$ such that 
\[
\begin{cases}
n_{11}^{-} \Phi(x_-,\lambda) + n_{12}^{-} \Phi'(x_-,\lambda) +  \Phi''(x_{-},\lambda)  = 0,\\
n_{21}^{-} \Phi(x_{-},\lambda) + n_{22}^{-} \Phi'(x_-,\lambda) + \Phi'''(x_{-},\lambda) = 0,
\end{cases}
\]
Hence $\Phi$ is solution of the ODE on $(x_-,x_+)$: 
 \[
\lambda^2(k^2\rho_0\Phi-(\rho_0\Phi')') +\lambda\mu(\Phi^{(4)}-2k^2\Phi''+k^4\Phi) =gk^2\rho_0'\Phi,
\]
with boundary conditions \eqref{GenLeftBound}, \eqref{GenRightBound}.

Conversely, assume that $\phi \in H^{4}((x_-,x_+))$ is a solution of equation
\[
\lambda^2(k^2\rho_0\phi-(\rho_0\phi')') +\lambda\mu(\phi^{(4)}-2k^2\phi''+k^4\phi) =gk^2\rho_0'\phi,
\]
 with boundary conditions \eqref{GenLeftBound}-\eqref{GenRightBound}.
 From the boundary conditions, we deduce that there exist $C_j^{+} (j=1,2), D_k^{-} ( k=3,4)$ such that 
\[
U(x_+, \lambda)=C_1^+U_1^+(x_+, \lambda)+C_2^+U_2^+(x_+, \lambda)
\]
 and 
 \[
 U(x_-, \lambda)=D_3^-U_3^-(x_-, \lambda)+D_4^-U_4^-(x_-, \lambda).
\] 
Then, through Cauchy-Lipschitz theorem, 
\[
U(x, \lambda)=C_1^+U_1^+(x, \lambda)+C_2^+U_2^+(x, \lambda) \quad\text{for all }x\geq x_+
\] 
and  
\[
U(x, \lambda)=D_3^-U_3^-(x, \lambda)+D_4^-U_4^-(x, \lambda) \quad\text{for all }x\leq x_-.
\]
 As these are decaying solutions at $\pm \infty$ respectively, and as there is no jump at $x_+$ (or $x_-$) for $\phi, \phi', \phi'', \phi'''$ (which have a meaning as $\phi$ is assumed to be in $H^{4}((x_-,x_+))$), the function 
\[
 \Phi(x)=
\begin{cases}
C_1^+U_{11}^+(x,\lambda)+C_2^+U_{21}^+(x,\lambda), \quad &\text{as } x\geq x_+ \\
\phi(x), \quad &\text{as }  x_-< x<x_+\\
 D_3^-U_{31}^-(x,\lambda)+D_4^-U_{41}^-(x,\lambda), \quad &\text{as } x\leq x_-
\end{cases} 
\]
belongs to $H^4(\R)$ and solves equation
\[
\lambda^2(k^2\rho_0\Phi-(\rho_0\Phi')') +\lambda\mu(\Phi^{(4)}-2k^2\Phi''+k^4\Phi) =gk^2\rho_0'\Phi,
\] on $\R$.
\end{proof}
Remark that in this proof we have no additional information on the values of $x_-^0$ and of $x_+^0$. This will be the aim of Section \ref{RefinedEstimates}.

\subsection{A bilinear form and a self-adjoint invertible operator}\label{SubSolInterval}
The aim of Section \ref{SubSolInterval} is to study the boundary terms that come from the bilinear form (\ref{2ndBilinearForm}) below.
Let
 \begin{equation}\label{BV_x+}
\begin{split}
BV_{x_+,\lambda}(\vartheta,\varrho) &:= -\lambda (\rho_0\vartheta' \varrho)(x_+) -\mu( n_{21}^+(x_+,\lambda) \vartheta(x_+) + n_{22}^+(x_+,\lambda) \vartheta'(x_+))\varrho(x_+) \\
&\quad+ \mu(n_{11}^+(x_+,\lambda) \vartheta(x_+)+ n_{12}^+(x_+,\lambda) \vartheta'(x_+))\varrho'(x_+) - 2k^2 \mu (\vartheta' \varrho)(x_+) 
\end{split}
\end{equation}
and by
\begin{equation}\label{BV_x-}
\begin{split}
BV_{x_-,\lambda}(\vartheta,\varrho) &:= \lambda (\rho_0\vartheta' \varrho)(x_-) + \mu( n_{21}^-(x_-,\lambda) \vartheta(x_-) + n_{22}^- (x_-,\lambda)\vartheta'(x_+))\varrho(x_-) \\
&\quad- \mu(n_{11}^-(x_-,\lambda) \vartheta(x_-) + n_{12}^-(x_-,\lambda) \vartheta'(x_-))\varrho'(x_-) + 2k^2 \mu (\vartheta'\varrho)(x_-). 
\end{split}
\end{equation}

\begin{lemma}
Necessary and sufficient conditions to get 
\begin{equation}\label{PositiveBV}
BV_{x_+,\lambda}(\vartheta,\vartheta)   \geq 0 \text{ and } BV_{x_-,\lambda}(\vartheta,\vartheta)   \geq 0
\end{equation}
 for all $\vartheta\in H^2([x_-, x_+])$ are
\begin{equation}
\label{ineq+}
\begin{split} 
&n_{12}^+(x_+,\lambda)\geq 0, \quad n_{21}^+(x_+,\lambda)\leq 0,  \\
& (n_{11}^+(x_+,\lambda)-n_{22}^+(x_+,\lambda)- k^2-\sigma_0^2(x_+,\lambda))^2+4 n_{12}^+n_{21}^+(x_+,\lambda)\leq 0,
 \end{split}
\end{equation}
and 
\begin{equation}
\label{ineq-}
\begin{split} 
& n_{12}^-(x_-,\lambda) \geq 0, \quad n_{21}^-(x_-,\lambda) \leq 0, \\
& (n_{11}^-(x_-,\lambda)  -n_{22}^-(x_-,\lambda) -k^2-\sigma_0^2(x_-,\lambda))^2+4 n_{12}^-n_{21}^-(x_-,\lambda)\leq 0.
 \end{split}
\end{equation}
\label{CoercivityLemma}
\end{lemma}

\begin{proof}[Proof of Lemma \ref{CoercivityLemma}]
We treat only the case $BV_{x_+,\lambda}(\vartheta,\vartheta)\geq 0$. Since $\mu \sigma_0^2(x,\lambda) =\mu k^2 +\lambda\rho_0(x)$, we rewrite 
\[
\begin{split}
\frac1{\mu} BV_{x_+,\lambda}(\vartheta,\vartheta) &= -n_{21}^+(x_+,\lambda)|\vartheta(x_+)|^2  + n_{12}^+|\vartheta'(x_+)|^2 \\
&\quad + (n_{11}^+(x_+,\lambda) -n_{22}^+(x_+,\lambda) -k^2-\sigma_0^2(x_+,\lambda)) \vartheta(x_+)\vartheta'(x_+).
\end{split}
\]
We observe that it is a quadratic polynomial in $\vartheta (x_+), \vartheta'(x_+)$. The first case is the case where $n_{12}^+=n_{21}^+=0$. The inequality (\ref{ineq+}) imply that $BV_{x_+,\lambda}(\vartheta,\vartheta)= 0$ for all $\vartheta$. The second case is the case where at least one of these two real numbers is not zero. For example, if $n_{12}^+\neq0$, the inequality means that the polynomial 
\[
 n_{12}^+ t^2+ (n_{11}^+-n_{22}^+-k^2-\sigma_0^2)t- n_{21}^+
\]
is always of the sign of $n_{12}^+$ hence positive, hence $BV_{x_+,\lambda}(\vartheta,\vartheta)\geq 0$ for all $\vartheta\in H^2((x_-,x_+)$. Lemma \ref{CoercivityLemma} is proven.
\end{proof}

\begin{proposition}\label{PropInverseGeneralT}
There exists $x_-^1\leq x_-^0$, $x_+^1\geq x_+^0$ such that, for any $x_-\leq x_-^1$ and $x_+\geq x_+^1$ chosen arbitrarily
\begin{equation}\label{2ndBilinearForm}
\begin{split}
\bB_{x_-,x_+,\lambda}(\vartheta,\varrho):=&BV_{x_-,\lambda}(\vartheta,\varrho) + BV_{x_+,\lambda}(\vartheta,\varrho)  + \lambda \int_{x_-}^{x_+}  \rho_0 (k^2\vartheta  \varrho + \vartheta'  \varrho') dx \\
&\qquad+  \mu \int_{x_-}^{x_+} (\vartheta''  \varrho'' + 2k^2 \vartheta'  \varrho' +k^4 \vartheta  \varrho)dx, 
\end{split}
\end{equation}
is a bilinear form on $H^2((x_-,x_+))$, that is continuous and coercive. 
\end{proposition}
\begin{proof}

Recall $n_{ij}^{\pm}$ $ (i,j=1,2)$ are given in Proposition \ref{PropEquivalentEq} and that $BV_{x_-,\lambda}$, $BV_{x_+,\lambda}$ are given in \eqref{BV_x+}, \eqref{BV_x-}. 
 One observes that one needs to prove that  
\begin{equation}\label{ContinuityB}
|\bB_{x_-,x_+,\lambda}(\vartheta,\varrho)| \leq C^{\star}  \| \vartheta\|_{H^2((x_-, x_+))}\| \varrho\|_{H^2((x_-, x_+))}
\end{equation}
and that
\begin{equation}\label{CoerciveB}
\bB_{x_-,x_+,\lambda}(\vartheta,\vartheta)\geq C^{\star} \| \vartheta\|^2_{H^2((x_-, x_+))}.
\end{equation}

For positive $\lambda, \mu$ and $k$, we have 
\[
\begin{split}
\lambda \int_{x_-}^{x_+}  \rho_0 (k^2\vartheta  \varrho + \vartheta'  \varrho') dx_3 &+  \mu \int_{x_-}^{x_+} (\vartheta''  \varrho'' + 2k^2 \vartheta'  \varrho' +k^4 \vartheta  \varrho)dx\\
& \leq C^{\star} \| \vartheta\|_{H^2((x_-, x_+))}\| \varrho\|_{H^2((x_-, x_+))}
\end{split}
\]
and that 
\[
\lambda \int_{x_-}^{x_+}  \rho_0 (k^2\vert \vartheta \vert^2 +\vert  \vartheta' \vert^2) dx +  \mu \int_{x_-}^{x_+} (\vert \vartheta'' \vert^2+ 2k^2\vert  \vartheta' \vert^2 +k^4 \vert \vartheta\vert^2)dx \geq C^{\star} \|\vartheta\|_{H^2((x_-,x_+))}^2.
\]
Note that the condition $\mu>0$ is necessary to obtain the coercivity on $H^2((x_-, x_+))$. Note also that the case $\mu=0$ amounts to the inviscid Rayleigh-Taylor instability, for which similar results are known (and the corresponding problem needs only to be defined in $H^1$).
In addition 
\[\begin{split}
\vert BV_{x_+,\lambda}(\vartheta,\varrho) \vert &\leq C^\star(|\vartheta(x_+)|+|\varrho(x_+)|)(|\vartheta'(x_+)|+|\varrho'(x_+)|), \\
|BV_{x_-,\lambda}(\vartheta,\varrho) \vert &\leq C^\star(|\vartheta(x_-)|+|\varrho(x_-)|)(|\vartheta'(x_-)|+|\varrho'(x_-)|)
 \end{split}
\]
and the Sobolev embedding  yields \eqref{ContinuityB}. The continuity of $\bB_{x_-,x_+,\lambda}$ on $H^2((x_-, x_+))$ follows.

To show \eqref{CoerciveB}, it suffices to prove that \eqref{PositiveBV} holds. In view of Lemma \ref{CoercivityLemma}, we verify \eqref{ineq+}. Since $N^+ U_1^+(x_+,\lambda)=N^+U_2^+(x_+,\lambda)=0$, we have that $n_{ij}^+$ $(i,j=1,2)$ depend on $x_+$ and $\lambda$ and satisfy
\begin{equation}\label{SystemN11_N12}
\begin{cases}
n_{11}^{+} U_{11}^+(x_+,\lambda) + n_{12}^{+} U_{12}^+(x_+,\lambda)+  U_{13}^+(x_+,\lambda)  = 0,\\
n_{11}^{+} U_{21}^+(x_+,\lambda) + n_{12}^{+} U_{22}^+(x_+,\lambda) + U_{23}^+(x_+,\lambda) = 0,
\end{cases}
\end{equation}
and 
\begin{equation}\label{SystemN21_N22}
\begin{cases}
n_{21}^{+} U_{11}^+(x_+,\lambda) + n_{22}^{+} U_{12}^+(x_+,\lambda)+  U_{14}^+(x_+,\lambda)  = 0,\\
n_{21}^{+} U_{21}^+(x_+,\lambda) + n_{22}^{+} U_{22}^+(x_+,\lambda) + U_{24}^+(x_+,\lambda) = 0.
\end{cases}
\end{equation}
Let $\cn_{ij}^+(\lambda)$ be the limit of $n_{ij}^+(x_+,\lambda)$ as $x_+\to+\infty$. When $x_+\rightarrow +\infty$, (\ref{SystemN11_N12})-(\ref{SystemN21_N22}) converge to
\begin{equation}\begin{cases}
-\cn_{11}^{+}(\lambda) k^{-3} + \cn_{12}^{+}(\lambda) k^{-2}-k^{-1}  = 0,\\
-\cn_{11}^{+}(\lambda) \sigma_+^{-3}(\lambda)+ \cn_{12}^{+}(\lambda) \sigma_+^{-2}(\lambda) - \sigma_+^{-1}(\lambda)= 0,
\end{cases}
\end{equation}
and 
\begin{equation}\begin{cases}
- \cn_{21}^{+}(\lambda) k^{-3} +\cn_{22}^{+}(\lambda) k^{-2}+1  = 0,\\
- \cn_{21}^{+}(\lambda) \sigma_+^{-3}(\lambda)+ \cn_{22}^{+}(\lambda) \sigma_+^{-2}(\lambda) +1= 0,
\end{cases}
\end{equation}
hence 
\[
\cn_{11}^+(\lambda) = k\sigma_+(\lambda), \quad \cn_{12}^+(\lambda)= k+\sigma_+(\lambda)
\]
and
\[
\cn_{21}^+(\lambda) = -k\sigma_+(\lambda)(k+\sigma_+(\lambda)), \quad \cn_{22}^+(\lambda)=-(k^2+k\sigma_+(\lambda)+\sigma_+^2(\lambda)).
\]
One thus has
\[\begin{split}
&(\cn_{11}^+(\lambda)-\cn_{22}^+(\lambda)-k^2-\sigma_+^2(\lambda))^2+4\cn_{12}^+\cn_{21}^+(\lambda) \\
&=((k+\sigma_+(\lambda))^2 -k^2 -\sigma_+^2(\lambda))^2 -4k\sigma_+(\lambda)(k+\sigma_+(\lambda))^2 \\
&=-4k\sigma_+(\lambda)(k^2+k\sigma_+(\lambda)+\sigma_+^2(\lambda))<0.
\end{split}\] 
Hence, by continuity, there exists $x_+^1\geq x_+^0$ such that, for all $x_+\geq x_+^1$, 
\[
(n_{11}^+(x_+,\lambda)-n_{22}^+(x_+,\lambda)-k^2-\sigma_0^2(x_+,\lambda))^2+4 n_{12}^+n_{21}^+(x_+,\lambda) <0.
\] 
The proof of the existence of $x_-^1\leq x_-^0$ such that, for all $x_-\leq x_-^1$ such that  
\[
(n_{11}^-(x_-,\lambda)  -n_{22}^-(x_-,\lambda) -k^2-\sigma_0^2(x_-,\lambda))^2+4 n_{12}^-n_{21}^-(x_-,\lambda) <0
\]
follows the same pattern. Hence, by application of Lemma \ref{CoercivityLemma},
\eqref{CoerciveB} follow,
which ends the proof of Proposition \ref{PropInverseGeneralT}.
\end{proof}

Mimicking the arguments in Propositions \ref{PropPropertyR}, \ref{PropInverseOfR}, \ref{RemNormR} we  obtain the  following proposition. 

\begin{proposition}\label{PropPropertyB_lambda}
Let  $(H^2((x_-,x_+)))'$ be the dual space of $H^2((x_-,x_+))$ associated with the norm $\sqrt{\bB_{x_-,x_+,\lambda}(\cdot,\cdot)}$, there exists a unique operator  
\[Y_{x_-,x_+,\lambda} \in  \mathcal{L}(H^2((x_-,x_+)), (H^2((x_-,x_+))'),\] that is also bijective,  such that
 \begin{equation}\label{EqMathcalB_lambda}
\bB_{x_-,x_+,\lambda}(\vartheta, \varrho) = \langle Y_{x_-,x_+,\lambda}\vartheta,  \varrho\rangle
\end{equation}
for all $\vartheta, \varrho \in H^2((x_-,x_+))$. Furthermore, we have 
\begin{enumerate}
\item For all $\vartheta \in H^2((x_-,x_+))$, 
\[
Y_{x_-,x_+,\lambda}\vartheta=\lambda (k^2\rho_0\vartheta-(\rho_0\vartheta')')+\mu(\vartheta^{(4)}-2k^2\vartheta''+k^4\vartheta)
\]  in $ \mathcal{D}'((x_-,x_+))$.  

\item Let $f\in L^2((x_-,x_+))$ be given, there exists a unique  $\vartheta \in H^4((x_-,x_+))$ satisfying the boundary conditions \eqref{GenLeftBound}--\eqref{GenRightBound}.\begin{equation}\label{EqZ=f}
Y_{x_-,x_+,\lambda}\vartheta = f \text{ in } ( H^2((x_-,x_+)))',
\end{equation}

\item The operator $Y_{x_-,x_+,\lambda}^{-1} : L^2((x_-,x_+)) \to L^2((x_-,x_+))$ is compact and self-adjoint.
\end{enumerate}
\end{proposition}

It is then straightforwad to obtain the following spectral result. The discrete spectrum of $\cM Y_{x_-,x_+,\lambda}^{-1}\cM$  is  a sequence of eigenvalues $(\gamma_n(\lambda))_{n\geqslant 1}$. The function $\gamma_n(\lambda)$ is a continuous function of $\lambda$ for all $n$ from the arguments of Lemma \ref{LemGammaCont}. The problem of finding a characteristic value  amounts to solving the equality 
\eqref{EqFindLambda} as before.  However, no control is possible on $\gamma_n(\lambda)$ when $\lambda$ goes to 0. In addition, no control of $x_\pm^0$ and $x_\pm^1$ (because it may depend on $\lambda$) is available to have a possibility of having estimates of $\gamma_n(\lambda)$ as well. Having an explicit (even if not optimal) criterion on $x_\pm^1$ such that the inequalities of coercivity \eqref{ineq+}-\eqref{ineq-} are true is the aim of the refined estimates of the solutions $U_{1,2}^+$ and $U_{3,4}^-$ (deducing in Propositions \ref{PropSolUplusInfty}, \ref{PropSolUminusInfty} below) which follow. That will be postponed to Section \ref{RefinedEstimates} below.

\subsection{The finding of characteristic value $\lambda_n$ and Proof of Theorem \ref{MainThm2}}\label{SectProfThm2}

Let $\epsilon_\star>0$ be given, we look for $\lambda_n \in (\epsilon_\star, \sqrt{\frac{g}{L_0}})$ satisfying \eqref{EqFindLambda}. However, unlike the previous case with $\rho_0' \geq 0$ being compactly supported, we do not have here the decrease of $\gamma_n$ on $\lambda$ to obtain the uniqueness of $\lambda_n$.

\begin{proposition}\label{PropFiniteLambda}
For $0<\epsilon_{\star}\ll 1$, there exists $N(\epsilon_{\star})\in \N$  such that  there is at least one positive $\lambda_n \in (\epsilon_{\star}, \sqrt{\frac{g}{L_0}})$ satisfying \eqref{EqFindLambda} for each $1\leq n\leq N(\epsilon_{\star})$. 
\end{proposition}
\begin{proof}
 We still have
\begin{equation}\label{2LimitgL_0}
\lim_{\lambda \to \sqrt{\frac{g}{L_0}}} \frac{\lambda}{\gamma_n(\lambda)} > gk^2.
\end{equation}
and  $b_n(\epsilon_{\star}) := \inf_{\lambda \geqslant \epsilon_{\star}} \gamma_n(\lambda)> 0$. Notice that $\{b_n(\epsilon_{\star})\}_{n\geqslant 1}$ is a sequence decreasing to $0$ as $n\to \infty$.  Set
\[
N(\epsilon_{\star}) := \sup \Big\{n| b_n(\epsilon_{\star}) > \frac{\epsilon_{\star}}{gk^2} \Big\} \in [1, +\infty).
\]
For $1\leq n\leq N(\epsilon_{\star})$,
\begin{equation}\label{GammaEpsilonStar}
\lim_{\lambda\to \epsilon_{\star}} \frac{\lambda}{\gamma_n(\lambda)} \leq \lim_{\lambda\to \epsilon_{\star}} \frac{\lambda}{b_n(\epsilon_{\star})} = \frac{\epsilon_{\star}}{b_n(\epsilon_{\star})} < gk^2.
\end{equation}

It then follows from \eqref{2LimitgL_0} and \eqref{GammaEpsilonStar} that  we have at least one desired $\lambda_n \in (\epsilon_{\star}, \sqrt{\frac{g}{L_0}})$ for $1\leq n\leq N(\epsilon_{\star})$.
\end{proof}

Now, we are able to prove Theorem \ref{MainThm2}.
\begin{proof}
Let $\varpi_n$ be an eigenfunction associated with $\gamma_n$ of $\cM Y_{x_+,x_-,\lambda}^{-1}\cM$.
That means 
\[
\cM Y_{x_-,x_+,\lambda_n}^{-1}\cM \varpi_n = \gamma_n(\lambda_n,k)\varpi_n = \frac{\lambda_n}{gk^2} \varpi_n.
\]
Hence, $\phi_n=  Y_{x_-,x_+,\lambda_n}^{-1}\cM \varpi_n \in H^4((x_-,x_+))$ satisfies
\[
\lambda_n Y_{x_-,x_+,\lambda_n} \phi_n = gk^2 \rho_0' \phi_n
\]
on $(x_-,x_+)$. In order to conclude that $\phi_n$ is a solution of \eqref{4thOrderEqPhi},  we then extend $\phi_n$ on $\R$ by continuity.

Let us take $\lambda =\lambda_n$ in the formulas of $U_{1,2}^+$ from \eqref{defiU+} and in the formulas of $U_{3,4}^-$ from \eqref{defiU-}.
Hence,  $\phi_n$ is of the form
\[
\phi_n(x) = B_{n,1}^+ U_{11}^+(x,\lambda_n) +  B_{n,2}^+ U_{21}^+(x,\lambda_n)
\]
as $x\geqslant x_+$ and 
\[
\phi_n(x) = B_{n,3}^- U_{31}^-(x,\lambda_n) +  B_{n,4}^- U_{41}^-(x,\lambda_n)
\]
as $x\leqslant x_-$  for some real constants $B_{n,1}^+, B_{n,2}^+, B_{n,3}^-$ and  $B_{n,4}^-$.  The constants $B_{n,1}^+, B_{n,2}^+$ are defined by 
\begin{equation}\label{EqB+}
\begin{cases}
\phi_n(x_+) = B_{n,1}^+ U_{11}^+(x_+,\lambda_n) +  B_{n,2}^+ U_{21}^+(x_+,\lambda_n) \\
\phi_n'(x_+) = B_{n,1}^+ U_{12}^+(x_+,\lambda_n) +  B_{n,2}^+  U_{22}^+(x_+,\lambda_n).
\end{cases}
\end{equation}
Similarly, we have the system for $B_{n,3}^-$ and $B_{n,4}^-$ is 
\begin{equation}\label{EqB-}
\begin{cases}
\phi_n(x_-) = B_{n,3}^-  U_{31}^-(x_-,\lambda_n)+  B_{n,4}^- U_{41}^-(x_-,\lambda_n), \\
\phi_n'(x_-) = B_{n,3}^- U_{32}^-(x_-,\lambda_n) +  B_{n,4}^-  U_{42}^-(x_-,\lambda_n).
\end{cases}
\end{equation}
Since $\phi_n$ is not trivial, there exists $i\geqslant 0$ such that  $\phi_n^{(i)}(x_+)$ is  nonzero. It implies that, at least, one of $B_{n,1}^+$ and $B_{n,2}^+$ is nonzero.

Solving \eqref{EqB+} and \eqref{EqB-}, we obtain that  
\[
B_{n,1}^+ = \frac{U_{22}^+(x_+,\lambda_n) \phi_n(x_+) - U_{21}^+(x_+,\lambda_n)\phi_n'(x_+)}{(U_{11}^+ U_{22}^+ - U_{12}^+U_{21}^+)(x_+,\lambda_n)}, 
\]
that
\[
B_{n,2}^+ =  \frac{-U_{12}^+(x_+,\lambda_n) \phi_n(x_+) +U_{11}^+(x_+,\lambda_n)\phi_n'(x_+)}{(U_{11}^+ U_{22}^+ - U_{12}^+U_{21}^+)(x_+,\lambda_n)},
\]
that
\[
B_{n,3}^- = \frac{U_{42}^-(x_-,\lambda_n)\phi_n(x_-) - U_{41}^-(x_-,\lambda_n)\phi_n'(x_-)}{(U_{31}^- U_{42}^- - U_{41}^- U_{32}^-)(x_-,\lambda_n)},
\]
and that
\[
B_{n,4}^- =\frac{-U_{32}^-(x_-,\lambda_n)\phi_n(x_-)+ U_{31}^-(x_-,\lambda_n)\phi_n'(x_-)}{(U_{31}^- U_{42}^- - U_{41}^- U_{32}^-)(x_-,\lambda_n)}.
\]
Therefore, we get that $\phi_n$ is a regular solution of \eqref{4thOrderEqPhi} as $\lambda =\lambda_n$ for each $1\leqslant n\leqslant N(\epsilon_{\star})$. This ends the proof of Theorem \ref{MainThm2}.
\end{proof}

\begin{remark}
In the forthcoming work \cite{Tai22_2}, we will use the above operator approach and a more precise hypothesis on $\rho_0-\rho_{\pm}$ to ensure that the classical Rayleigh-Taylor instability has an infinite number of discrete spectrum. That will complete the study in \cite{Laf01}, where the first author only considered the case $\frac{\rho_0'}{\rho_0}$ has a nondegenerate minimum.
\end{remark}

\begin{remark}
\begin{enumerate}
\item For $|x_{\pm}|$ large enough, the investigation of regular solutions to Eq. \eqref{4thOrderEqPhi} on the real line is equivalent to that one on $(x_-,x_+)$ with boundary conditions \eqref{GenLeftBound} and \eqref{GenRightBound} at $x_{\pm}$.
More computations are required in the second case due to the lack of compact assumption of $\rho_0'$. 
\item For $|x_{\pm}|$ large enough, the problem \eqref{4thOrderEqPhi} on $(x_-,x_+)$ with boundary conditions \eqref{GenLeftBound} and \eqref{GenRightBound} is equivalent to a weaker version of \eqref{4thOrderEqPhi} that can be solved by  applying Riesz representation theorem on the bilinear continuous form $\bB_{x_-,x_+,\lambda}$ and then improving the regularity. 
\item Generally speaking, a problem on the real line with decaying solutions at infinity and transversality hypotheses  is equivalent to a problem with the compact setting when we have enough decays on the solutions at $\pm \infty$.
\end{enumerate}
\end{remark}

\subsection{Explicit construction and refined estimates of the decaying solutions  at infinity}\label{RefinedEstimates}

In the regime $\lambda\geq \epsilon_\star>0$, we notice again that $R(\lambda)$ is uniformly bounded.   So that, as $\lambda\geq \epsilon_\star>0$, we further  derive a control of the inner region $(x_-,x_+)$ independent of $\lambda$ in the following proposition, extending the result of Proposition \ref{PropInverseGeneralT}.
\begin{proposition}
Let $\epsilon_\star>0$ given and let $\lambda \geq \epsilon_\star$. Let $z_{+,\epsilon_\star}$ (respectively, $z_{-,\epsilon_\star}$) be the sum of two upper bounds, that are functions decreasing towards 0 at $+\infty$ (respectively, at $-\infty$), on the r.h.s. of \eqref{DecayU1} and \eqref{DecayU2} (respectively, \eqref{DecayU3} and \eqref{DecayU4}). There exist positive constants $\Gamma_{\pm}( \epsilon_\star)$ such that, for all $x_+, x_-$ such that 
\begin{equation}
z_{+,\epsilon_\star}(x_+)\leq \Gamma_+(\epsilon_\star)\quad\text{and}\quad  z_{-,\epsilon_\star}(x_-)\leq \Gamma_-(\epsilon_\star),
\end{equation}
we have  $\bB_{x_-,x_+,\lambda}$ is coercive.
\label{coercive-interval}
\end{proposition}

The proof of Proposition \ref{coercive-interval} relies on the refined estimates of the bounded solutions of \eqref{EqDiffU} near $\infty$, presented in Propositions \ref{PropSolUplusInfty}, \ref{PropSolUminusInfty}. Before going to the proof of Propositions \ref{PropSolUplusInfty}, \ref{PropSolUminusInfty}, thus Proposition \ref{coercive-interval}, we present some materials.  Notice from \eqref{MatrixL} that one has $L(x,\lambda)=P(x,\lambda) D(x,\lambda) P(x,\lambda)^{-1}$, where
\begin{equation}\label{MatrixD}
D(x,\lambda)= \text{diag} ( -k , -\sigma_0(x,\lambda), k, \sigma_0(x,\lambda) ),
\end{equation}
\begin{equation}\label{MatrixP}
P(x,\lambda) = \begin{pmatrix}
-k^{-3} & -\sigma_0^{-3}(x,\lambda) & k^{-3} & \sigma_0^{-3}(x,\lambda) \\
k^{-2} & \sigma_0^{-2}(x,\lambda) &  k^{-2} & \sigma_0^{-2}(x,\lambda) \\
-k^{-1} & -\sigma_0^{-1}(x,\lambda) & k^{-1} & \sigma_0^{-1}(x,\lambda)  \\
1 & 1 & 1 & 1
\end{pmatrix}
\end{equation}
and 
\begin{equation}\label{MatrixPinverse}
P(x,\lambda)^{-1} = \frac{\mu}{2\lambda \rho_0(x)} 
\begin{pmatrix}
-k^3 \sigma_0^2(x,\lambda) & k^2 \sigma_0^2(x,\lambda) & k^{3} & -k^2 \\
k^2 \sigma_0^3(x,\lambda)& -k^2 \sigma_0^{2}(x,\lambda) &  -\sigma_0^3(x,\lambda) & \sigma_0^2(x,\lambda)\\
k^3 \sigma_0^2(x,\lambda) & k^2 \sigma_0^2(x,\lambda)& -k^3 & -k^2  \\
-k^2 \sigma_0^3(x,\lambda) & -k^2 \sigma_0^2(x,\lambda) & \sigma_0^3(x,\lambda) & \sigma_0^2(x,\lambda)
\end{pmatrix}
\end{equation}
The columns of matrix $P$ are denoted by $P_1, P_2,P_3,P_4$ and $P_2,P_4$ depend on $(x, \lambda)$. Note that for every positive $k,\lambda $ and $\mu$, $P$ and $P^{-1}$ are bounded uniformly in $\R$ and  $P^{-1}$ becomes singular when $\lambda \to 0$ and $x$ is fixed. 

Then, we set $U(x)=P(x,\lambda)V(x)$, \eqref{EqDiffU} becomes 
\begin{equation}\label{EqDiffV}
V'(x) = (D(x,\lambda)+ \rho_0'(x) M(x,\lambda)) V(x),
\end{equation}
where
\[
M(x,\lambda) = P(x,\lambda)^{-1}R(\lambda)P(x,\lambda) - \frac{d\sigma_0(x)}{d\rho_0(x)}P(x,\lambda)^{-1} \frac{dP(x,\lambda)}{d\sigma_0(x)}.
\]
\begin{lemma}\label{LemBoundGamma}
Let
\[
\delta(\epsilon_\star) := \sqrt{k^2 + \frac{\epsilon_\star \rho_-}{\mu}} \quad\text{and}\quad \delta_s :=\sqrt{k^2 + \sqrt{\frac{g}{L_0}} \frac{\rho_+}{\mu}}.
\]
For any $\lambda \in [\epsilon_\star,\sqrt{\frac{g}{L_0}}]$, there hold 
\begin{equation}\label{BoundGammaP}
\sup_{x\in \R} \|P(x,\lambda)\|_F \leqslant \Gamma_p := \max\Big(1,\frac1k, \frac1{k^2}, \frac1{k^3}\Big) 
\end{equation}
and 
\begin{equation}\label{BoundGammaM}
\begin{split}
\sup_{x\in \R} \|M(x,\lambda)\|_F \leqslant \Gamma_m(\epsilon_\star) &:= \frac1{ \rho_- \epsilon_\star^2} \max\Big( g\Big(k + \frac1{L_0}\Big),  g\Big(\frac{k^2}{\delta(\epsilon_\star)}+ \frac1{L_0}\Big) \Big) \\
&\quad+ \frac1{4\delta(\epsilon_\star)} \sqrt{\frac{g}{L_0}} \max\Big( \frac{2k^2}{\delta^2(\epsilon_\star)}(k+\delta_s), \frac{5k^2}{\delta(\epsilon_\star)} +\delta_s,  \frac{k^2}{\delta(\epsilon_\star)}+\delta_s \Big).
\end{split}
\end{equation}
\end{lemma}
\begin{proof}
Due to $\epsilon_\star \leqslant \lambda \leqslant \sqrt{\frac{g}{L_0}}$, we get that 
$\delta(\epsilon_\star) < \sigma_0(x,\lambda) \leqslant  \delta_s$ for all $x\in \R$, it yields \eqref{BoundGammaP}. We move to demonstrate \eqref{BoundGammaM}. Let 
\[
a_{\pm}(\lambda) = gk \pm \lambda^2, \quad b_{\pm}(x,\lambda)= \frac{gk^2}{\sigma_0(x,\lambda)} \pm \lambda^2
\]
and
\[
  c_{\pm}(x,\lambda) = k\pm \sigma_0(x,\lambda), \quad d(x,\lambda) =\frac{5k^2}{\sigma_0(x,\lambda)}-\sigma_0(x,\lambda).
\]
Direct computations show that 
\[
\frac{d\sigma_0(x,\lambda)}{d\rho_0(x)} = \frac{\lambda}{2\mu} \frac{\rho_0'(x)}{\sigma_0(x,\lambda)}, 
\]
that 
\[
P(x,\lambda)^{-1}R(\lambda) P(x,\lambda) =  \frac{1}{2\lambda^2 \rho_0(x)} \begin{pmatrix}
a_- & \frac{k^2 b_-}{\sigma_0^2} &-a_+& - \frac{k^2  b_+}{\sigma_0^2}\\
-\frac{a_-\sigma_0^2}{k^2}& -b_- &\frac{a_+\sigma_0^2}{k^2}&-b_+ \\
a_- & \frac{k^2 b_-}{\sigma_0^2} &-a_+& - \frac{k^2  b_+}{\sigma_0^2}\\
-\frac{a_-\sigma_0^2}{k^2}& -b_-  &\frac{a_+\sigma_0^2}{k^2}&-b_+ \\
\end{pmatrix}(x,\lambda)
\]
and that 
\[
P(x,\lambda)^{-1}\frac{dP(x,\lambda)}{d\sigma_0(x,\lambda)} = \frac{\mu}{2\lambda \rho_0(x)} 
\begin{pmatrix}
 0&  -\frac{2k^2 c_+}{\sigma_0^2}& 0 & \frac{2k^2 c_-}{\sigma_0^2} \\
0 & d & 0 & -\frac{k^2}{\sigma_0} + \sigma_0 \\
 0&  \frac{2k^2 c_-}{\sigma_0^2}& 0 & -\frac{2k^2 c_+}{\sigma_0^2} \\
0 &  -\frac{k^2}{\sigma_0} + \sigma_0 & 0 & d \\
\end{pmatrix}(x,\lambda).
\]
For all $x\in \R$, it is clear that
\[
|a_{\pm}(\lambda)| \leqslant g\Big(k + \frac1{L_0}\Big),\quad |b_{\pm}(x,\lambda)| \leqslant g\Big(\frac{k^2}{\delta(\epsilon_\star)}+ \frac1{L_0}\Big)
\]
and that
\[
|c_{\pm}(x,\lambda)| \leqslant k +\delta_s, \quad |d(x,\lambda)| \leqslant \frac{5k^2}{\delta(\epsilon_\star)} +\delta_s.
\] 
Therefore, \eqref{BoundGammaM} follows. Proof of Lemma \ref{LemBoundGamma} is complete.
\end{proof}

Let $\tilde x_+$ be chosen such that 
\begin{equation}\label{tildeX+}
\int_{\tilde x_+}^{+\infty} \rho_0'(\tau)\|M(\tau,\lambda)\|_F d\tau \leq \Gamma_m(\epsilon_\star)(\rho_+ -\rho_0(\tilde x_+)) <\frac12
\end{equation}
and $\tilde x_-$ be chosen such that 
\begin{equation}\label{tildeX-}
\int_{-\infty}^{\tilde x_-} \rho_0'(\tau)\|M(\tau,\lambda)\|_F d\tau \leqslant \Gamma_m(\epsilon_\star)(\rho_0(\tilde x_-)-\rho_-)<\frac12.
\end{equation}
Let $\alpha_{\pm}(x)$ and $\beta_{\pm}(x,\lambda)$ be defined
\begin{equation}\label{EqAlphaBeta}
\alpha_{\pm}(x)=\pm k(x-{\tilde x}_{\pm}),\quad \beta_{\pm}(x, \lambda)=\pm \int_{{\tilde x}_{\pm}}^x \sigma_0(y,\lambda)dy.
\end{equation}
We then study the solutions of \eqref{EqDiffV} decaying to 0  at $+\infty$.
\begin{proposition}\label{LemSolVplusInfty}
Eq. \eqref{EqDiffV} on $(\tilde x_+, +\infty)$ admits a unique solution $V_1(x,\lambda)$ such that $e^{-\alpha_+(x)}V_1(x)$ converges  to $e_1$ as $x\to+\infty$ and a unique solution  $V_2(x,\lambda)$  such that $e^{-\beta_+(x,\lambda)}V_2(x)$ converges  to $e_2$ as $x\to+\infty$. Furthermore, we have the following estimates 
\begin{equation}\label{V1LimitUniform}
\|e^{\alpha_+(x)}V_1(x,\lambda)  -e_1\|_2 \leqslant 2\Gamma_m(\epsilon_\star) 
\left( \begin{split}
&(\rho_+-\rho_0(x)) +\rho_0(\tilde x_+) e^{-(\delta(\epsilon_\star)-k)(x-\tilde x_+)}\\
&+\Big| \rho_0(x)   -  (\delta(\epsilon_\star)-k) \int_{\tilde x_+}^x \rho_0(\tau) e^{-(\delta(\epsilon_\star)-k)(x-\tau)}d\tau \Big|
\end{split}\right)
\end{equation}
and 
\begin{equation}\label{V2LimitUniform}
\|e^{\beta_+(x,\lambda)} V_2(x,\lambda) -e_2\|_2 \leqslant 2\Gamma_m(\epsilon_\star) (\rho_+-\rho_0(x))
\end{equation}
for all $x\geqslant \tilde x_+$. 
\end{proposition}
\begin{proof}
We define the matrices
\[
\Psi(x,\lambda) = \text{diag}(e^{-\alpha_+(x)}, e^{-\beta_+(x,\lambda)}, e^{\alpha_+(x)}, e^{\beta_+(x,\lambda)}),
\]

\[
\Psi_1(x,\lambda) =  \text{diag}(0, e^{-\beta_+(x,\lambda)}, 0, 0)
\]
and 
\[
\Psi_2(x,\lambda) =  \text{diag}(e^{-\alpha_+(x)}, 0, e^{\alpha_+(x)}, e^{\beta_+(x,\lambda)}).
\]
Then, we consider the equation
\begin{equation}\label{OdeV-k}
\begin{split}
V_1(x,\lambda) &= e^{-\alpha_+(x)}e_1+ \int_{\tilde x_+}^{x} \Psi_1(x,\lambda) \Psi^{-1}(\tau,\lambda) \rho_0'(\tau) M(\tau,\lambda) V_1(\tau,\lambda)d\tau\\ 
&\qquad-\int_{x}^{+\infty} \Psi_2(x,\lambda)\Psi(\tau,\lambda)^{-1} \rho_0'(\tau) M(\tau,\lambda) V_1(\tau,\lambda)d\tau.
\end{split}
\end{equation}
It can be seen that a solution $V_1$ of \eqref{OdeV-k} satisfies  \eqref{EqDiffV}. We solve \eqref{OdeV-k} by the Picard iteration method. Indeed, let $V_1^{(0)}(x)=0$ and 
\[
\begin{split}
V_1^{(j+1)}(x,\lambda) &= e^{-\alpha_+(x)}e_1 + \int_{\tilde x_+}^{x} \Psi_1(x,\lambda) \Psi^{-1}(\tau,\lambda) \rho_0'(\tau) M(\tau,\lambda) V_1^{(j)}(\tau,\lambda)d\tau\\ 
&\qquad-\int_{x}^{+\infty} \Psi_2(x,\lambda)\Psi^{-1}(\tau,\lambda) \rho_0'(\tau) M(\tau,\lambda) V_1^{(j)}(\tau,\lambda)d\tau.
\end{split}
\]
We have that
\[
\Psi_1(x,\lambda) \Psi^{-1}(\tau,\lambda) = \text{diag}(0, e^{-(\beta_+(x,\lambda)-\beta_+(\tau,\lambda))}, 0, 0)
\]
and that
\[
\Psi_2(x,\lambda) \Psi^{-1}(\tau,\lambda)= \text{diag}(e^{-(\alpha_+(\tau)-\alpha_+(x))}, 0, e^{\alpha_+(\tau)-\alpha_+(x)}, e^{-(\beta_+(x,\lambda)-\beta_+(\tau,\lambda))}).
\]
 Hence,  we can estimate for $\tilde x_+ \leqslant \tau \leqslant x$, 
\begin{equation}\label{EqPsi1}
\begin{split}
\|\Psi_1(x,\lambda) \Psi^{-1}(\tau,\lambda)\|_F &\leqslant e^{-(\alpha_+(x)-\alpha_+(\tau)) + \int_{\tau}^{x} (k-\sigma_0(s)) ds }\\
&\leqslant e^{-(\alpha_+(x)-\alpha_+(\tau)) -(\delta(\epsilon_\star)-k)(x-\tau)}
\end{split}
\end{equation}
and for $\tau \geqslant x$, 
\begin{equation}\label{EqPsi2}
\|\Psi_2(x,\lambda) \Psi^{-1}(\tau,\lambda)\|_F \leqslant e^{-(\alpha_+(x)-\alpha_+(\tau))}.
\end{equation}
Using \eqref{EqPsi1} and \eqref{EqPsi2}, we get
\[
\begin{split}
 &e^{\alpha_+(x)}\|V_1^{(j+1)}(x,\lambda) -V_1^{(j)}(x,\lambda)\|_2\\
 &\qquad\leqslant \Gamma_m(\epsilon_\star) \int_{\tilde x_+}^{\infty}e^{\alpha_+(\tau)}   \rho_0'(\tau)
\|V_1^{(j)}(\tau,\lambda) -V_1^{(j-1)}(\tau,\lambda)\|_F d\tau.
\end{split}
\]
Thanks to the induction, we get for all $x\geqslant \tilde x_+$ and for all $j\geqslant 0$,
\begin{equation}\label{EstVj}
 e^{\alpha_+(x)} \|V_1^{(j+1)}(x,\lambda) -V_1^{(j)}(x,\lambda)\|_2 \leqslant \Big(\frac12\Big)^j,
\end{equation}
yielding the uniform convergence of $\{V_1^{(j)}(x,\lambda)\}_{j\geqslant 0}$ on any interval  of $(\tilde x_+,+\infty)$. Let $V_1(x,\lambda)$ be the limit function. $V_1^{(j)}(x,\lambda)$ is continuous, so is $V_1(x,\lambda)$. Moreover, \eqref{EstVj} implies that
\[
 e^{\alpha_+(x)} \|V_1^{(j+1)}(x,\lambda)\|_2  \leqslant \sum_{i=0}^j e^{\alpha_+(x)}\|V_1^{(i+1)}(x,\lambda) -V_1^{(i)}(x,\lambda)\|  \leqslant  \sum_{i=0}^j  \Big(\frac12\Big)^i.
\]
That tells us for $x\geq \tilde x_+$, 
\begin{equation}\label{BoundV(x)}
\|V_1(x,\lambda)\|_2 \leqslant 2 e^{-\alpha_+(x)}.
\end{equation}

Once we have \eqref{BoundV(x)}, we then prove \eqref{V1LimitUniform}. Indeed, 
\[
\begin{split}
 e^{\alpha_+(x)}V_1(x,\lambda) - e_1 &= e^{\alpha_+(x)} \int_{\tilde x_+}^{x} \Psi_1(x,\lambda) \Psi^{-1}(\tau,\lambda) \rho_0'(\tau) M(\tau,\lambda) V_1(\tau,\lambda)d\tau\\ 
&\qquad-e^{\alpha_+(x)}\int_{x}^{+\infty} \Psi_2(x,\lambda)\Psi(\tau,\lambda)^{-1}\rho_0'(\tau) M(\tau,\lambda)V_1(\tau,\lambda)d\tau.
\end{split}
\]
We make use of \eqref{EqPsi2} and \eqref{BoundV(x)} to have that
\begin{equation}\label{EstPsi2XtoInfty}
\begin{split}
e^{\alpha_+(x)} \int_{x}^{+\infty} \|\Psi_2(x,\lambda)\Psi^{-1}(\tau,\lambda) \rho_0'(\tau) M(\tau,\lambda) V_1(\tau,\lambda)\|_2 d\tau   
\leqslant 2\Gamma_m(\epsilon_\star)(\rho_+-\rho_0(x)).
\end{split}
\end{equation}
From \eqref{EqPsi1} and \eqref{BoundV(x)}, we obtain that 
\begin{equation}\label{EstBoundPsi1XtoX}
\begin{split}
\|e^{\alpha_+(x)} \int_{\tilde x_+}^{x} \Psi_1(x,\lambda) \Psi^{-1}(\tau,\lambda) \rho_0'(\tau)  &M(\tau,\lambda)V_1(\tau,\lambda)d\tau \|_2 \\
&\leqslant 2\Gamma_m(\epsilon_\star) \int_{\tilde x_+}^x \rho_0'(\tau) e^{-(\delta(\epsilon_\star)-k)(x-\tau)}d\tau.
\end{split}
\end{equation}
After integrating by parts, we get 
\begin{equation}\label{EqIntegrateBoundPsi1X}
\begin{split}
\int_{\tilde x_+}^x \rho_0'(\tau) e^{-(\delta(\epsilon_\star)-k)(x-\tau)}d\tau &= -\rho_0(\tilde x_+) e^{-(\delta(\epsilon_\star)-k)(x-\tilde x_+)}+ \rho_0(x)  \\
& \qquad -  (\delta(\epsilon_\star)-k) \int_{\tilde x_+}^x \rho_0(\tau) e^{-(\delta(\epsilon_\star)-k)(x-\tau)}d\tau.
\end{split}
\end{equation}
Combining \eqref{EstPsi2XtoInfty}, \eqref{EstBoundPsi1XtoX} and \eqref{EqIntegrateBoundPsi1X} gives  \eqref{V1LimitUniform}.

By considering the eigenvalue $-\sigma_0(x,\lambda)$ of $L(x,\lambda)$, we continue the idea in Theorem \ref{Theorem81} and mimic the above arguments to the solution $V_2(x,\lambda)$ such that  $e^{\beta_+(x,\lambda)}V_2(x)$ converges to $e_2$ at $+\infty$ and enjoying \eqref{V2LimitUniform}.
That ends the proof of Lemma \ref{LemSolVplusInfty}.
\end{proof}

Now, we get back to \eqref{EqDiffU} to find solutions that are bounded near $+\infty$. 
\begin{proposition}\label{PropSolUplusInfty}   
Eq. \eqref{EqDiffU} on $(\tilde x_+, +\infty)$ admits 
\begin{enumerate}
\item a unique solution $U_1^+(x,\lambda)$ satisfying that as $x\to +\infty$, $e^{\alpha_+(x)}U_2^+(x,\lambda)$ converges to $(-k^{-3},k^{-2},-k^{-1},1)^T$  and that for all $x\geq \tilde x_+$,
\begin{equation}\label{DecayU1}
\begin{split}
& \|e^{\alpha_+(x)}U_1^+(x,\lambda)  - (-k^{-3},k^{-2}, -k^{-1}, 1)^T \|_2 \\
&\qquad\leqslant 2\Gamma_p \Gamma_m(\epsilon_\star) \left( \begin{split}
&(\rho_+-\rho_0(x)) +\rho_0(\tilde x_+) e^{-(\delta(\epsilon_\star)-k)(x-\tilde x_+)}\\
&+\Big| \rho_0(x)   -  (\delta(\epsilon_\star)-k) \int_{\tilde x_+}^x \rho_0(\tau) e^{-(\delta(\epsilon_\star)-k)(x-\tau)}d\tau \Big|
\end{split}\right),
\end{split}
\end{equation}
\item  a unique solution $U_2^+(x,\lambda)$  satisfying that as $x\to+\infty$, $e^{\beta_+(x,\lambda)}U_2^+(x,\lambda)$ converges to $(-\sigma_+^{-3}(\lambda), \sigma_+^{-2}(\lambda), -\sigma_+^{-1}(\lambda),1)^T$  and that for all $x\geq \tilde x_+$,
\begin{equation}\label{DecayU2}
\begin{split}
&\| e^{\beta_+(x,\lambda)}U_2^+(x,\lambda) -(-\sigma_+^{-3}(\lambda), \sigma_+^{-2}(\lambda), -\sigma_+^{-1}(\lambda), 1)^T\|_2
 \\&\qquad \leqslant   \Big( \sqrt{ \frac{g(4\delta^{10}(\epsilon_\star)+16\delta^{12}(\epsilon_\star)+9\delta_s^4)}{16 L_0\mu^2 \delta^{16}(\epsilon_\star)}}+2\Gamma_p\Gamma_m(\epsilon_\star)\Big)  (\rho_+ -\rho_0(x)).
\end{split}
\end{equation}
\end{enumerate}
\end{proposition}
\begin{proof}
We define $U_j^+(x,\lambda) =P(x,\lambda)V_j(x,\lambda) (j=1,2)$, with $V_1$ and $V_2$ are two solutions of \eqref{EqDiffV} satisfying \eqref{V1LimitUniform} and \eqref{V2LimitUniform} respectively. It can be seen that $U_1^+(x,\lambda)$ and $U_2^+(x,\lambda)$ are two solution of \eqref{EqDiffU}. 

Note that 
\[
\begin{split}
 e^{\alpha_+(x)} U_1^+(x,\lambda) &= P_1 + P(x,\lambda)(e^{\alpha_+(x)}V_1(x,\lambda)  -e_1) \\
&= (-k^{-3},k^{-2}, -k^{-1}, 1)^T + P(x,\lambda)(e^{\alpha_+(x)} V_1(x,\lambda) -e_1).
\end{split}
\]
\eqref{DecayU1} is then clear due to the estimate \eqref{V1LimitUniform}.  According to l'H\^opital's rule,  we have 
\[
\lim_{x\to +\infty} \frac{\int_{\tilde x_+}^x \rho_0(\tau) e^{(\delta(\epsilon_\star)-k)\tau}d\tau}{e^{(\delta(\epsilon_\star)-k)x}} = \lim_{x\to +\infty} \frac{\rho_0(x) e^{(\delta(\epsilon_\star)-k)x}}{(\delta(\epsilon_\star)-k) e^{(\delta(\epsilon_\star)-k)x}} = \frac{\rho_+}{\delta(\epsilon_\star)-k},
\]
that implies
\[
\lim_{x\to +\infty}\Big| \rho_0(x)   -  (\delta(\epsilon_\star)-k) \int_{\tilde x_+}^x \rho_0(\tau) e^{-(\delta(\epsilon_\star)-k)(x-\tau)}d\tau \Big| =0.
\]
The behavior of $U_1^+(x,\lambda)$ at $+\infty$ follows.

To prove \eqref{DecayU2}, we write 
\[
\begin{split}
&e^{\beta_+(x,\lambda)} U_2^+(x,\lambda) - (-\sigma_+^{-3}(\lambda), \sigma_+^{-2}(\lambda), -\sigma_+^{-1}(\lambda), 1 )^T \\
&= P_2(x,\lambda) - (-\sigma_+^{-3}(\lambda), \sigma_+^{-2}(\lambda), -\sigma_+^{-1}(\lambda), 1 )^T + P(x,\lambda)(e^{\beta_+(x,\lambda)}V_2(x,\lambda)  -e_2).
\end{split}
\]
Since $\delta(\epsilon_\star) < \sigma_+(\lambda) <\delta_s$ for all $\lambda\in [\epsilon_\star,\sqrt{\frac{g}{L_0}}]$, we bound  that
\begin{equation}\label{NormP(x)E_2Minus}
\begin{split}
&\| P_2(x,\lambda) - (-\sigma_+^{-3}(\lambda), \sigma_+^{-2}(\lambda), -\sigma_+^{-1}(\lambda), 1)^T\|_2^2  \\
&= \frac{\lambda^2}{\mu^2}(\rho_0(x)-\rho_+)^2 
\left[ \begin{split}
 &\frac1{\sigma_0^2(x,\lambda)\sigma_+^2(\lambda) (\sigma_0(x,\lambda)+\sigma_+(\lambda))^2} + \frac1{\sigma_0^2(x,\lambda) \sigma_+^2(\lambda)} \\
 &\quad+ \Big( \frac{\sigma_0^2(x,\lambda)+\sigma_0(x,\lambda) \sigma_+(\lambda) + \sigma_+^2(\lambda)}{\sigma_0^3(x,\lambda)  \sigma_+^3(\lambda) (\sigma_0(x,\lambda)+\sigma_+(\lambda)) }\Big)^2 
\end{split} \right] \\
&\leqslant \frac{g(4\delta^{10}(\epsilon_\star)+16\delta^{12}(\epsilon_\star)+9\delta_s^4)}{16 L_0\mu^2 \delta^{16}(\epsilon_\star)} (\rho_0(x)-\rho_+)^2.
\end{split}
\end{equation}
Meanwhile, as a result of \eqref{V2LimitUniform}, 
\begin{equation}\label{NormP(x)V_2(x)}
\|P(x,\lambda)(e^{\beta_+(x,\lambda)}V_2(x,\lambda) -e_2)\|_2 \leqslant 2\Gamma_p \Gamma_m(\epsilon_\star)(\rho_+-\rho_0(x)).
\end{equation}
Thanks to \eqref{NormP(x)E_2Minus} and \eqref{NormP(x)V_2(x)}, we obtain \eqref{DecayU2} hence the behavior of $U_2^+(x,\lambda)$ at $+\infty$.
Proof of Proposition \ref{PropSolUplusInfty} is complete.
\end{proof}

 We now fix two positive eigenvalues of $L(x,\lambda)$, $k$ and $\sigma_0(x,\lambda)$ and thus follow Theorem \ref{Theorem81} again.  We are able to construct solutions of \eqref{EqDiffU} that are bounded near $-\infty$ as in Proposition \ref{PropSolUplusInfty}. 
\begin{proposition}\label{PropSolUminusInfty} 
Eq. \eqref{EqDiffU}  on $(-\infty, \tilde x_-)$ admits 
\begin{enumerate}
\item a unique solution $U_3^-(x,\lambda)$ satisfying that as $x\to -\infty$, $e^{\alpha_-(x)}U_3^-(x,\lambda)$ converges to $(k^{-3},k^{-2},k^{-1},1)^T$ and  that, for all  $x\leq \tilde x_-$
\begin{equation}\label{DecayU3}
\begin{split}
&\| e^{\alpha_-(x)}U_3^-(x,\lambda)  -  (k^{-3},k^{-2}, k^{-1}, 1)^T \|_2 \\ &\qquad\quad\leqslant 2\Gamma_p\Gamma_m(\epsilon_\star) \left( \begin{split}
&(\rho_0(x)-\rho_-)+ \rho_0(\tilde x_-) e^{-(\delta(\epsilon_\star)-k)(\tilde x_- -x)} \\
&+  \Big| \rho_0(x)- (\delta(\epsilon_\star)-k) \int_x^{\tilde x_-} \rho_0(\tau) e^{-(\delta(\epsilon_\star)-k)(\tau-x)}d\tau\Big| 
\end{split} \right),
\end{split}
\end{equation} 
\item a unique solution $U_4^-(x,\lambda)$ satisfying that as $x\to -\infty$, $e^{\beta_-(x,\lambda)}U_4^-(x,\lambda)$ converges to $(\sigma_-^{-3}(\lambda), \sigma_-^{-2}(\lambda), \sigma_-^{-1}(\lambda),1)^T$ and that for all $x\leq \tilde x_-$,
\begin{equation}\label{DecayU4}
\begin{split}
&\|e^{\beta_-(x,\lambda)}U_4^-(x,\lambda)  - (\sigma_-^{-3}(\lambda), \sigma_-^{-2}(\lambda), \sigma_-^{-1}(\lambda), 1)^T\|_2 \\
&\qquad\leqslant    \Big(  \sqrt{\frac{g(4\delta^{10}(\epsilon_\star)+16\delta^{12}(\epsilon_\star)+9\delta_s^4)}{16 L_0\mu^2 \delta^{16}(\epsilon_\star)}} +2\Gamma_p\Gamma_m(\epsilon_\star)\Big) (\rho_0(x)-\rho_-).
\end{split}
\end{equation} 
\end{enumerate}
\end{proposition}

We are now in position to prove Proposition \ref{coercive-interval}.
\begin{proof}
We recall $n_{ij}^+$ ($i,j=1,2$) from \eqref{SystemN11_N12} and \eqref{SystemN21_N22} to have that
\begin{equation}\label{EqN1byN34}
\begin{pmatrix} n_{11}^+ \\ n_{12}^+ \end{pmatrix}(x_+,\lambda)
= - \frac1{U_{11}^+ U_{22}^+ - U_{21}^+U_{12}^+ } 
\begin{pmatrix} U_{22}^+ & -U_{12}^+ \\ - U_{21}^+ & U_{11}^+ \end{pmatrix} 
\begin{pmatrix} U_{13}^+ \\ U_{23}^+ \end{pmatrix}(x_+,\lambda)
\end{equation}
and 
\begin{equation}\label{EqN2byN34}
\begin{pmatrix} n_{21}^+ \\ n_{22}^+ \end{pmatrix}(x_+,\lambda)
=- \frac1{U_{11}^+ U_{22}^+ - U_{21}^+U_{12}^+ } 
\begin{pmatrix} U_{22}^+ & -U_{12}^+ \\ - U_{21}^+ & U_{11}^+ \end{pmatrix} 
\begin{pmatrix}  U_{14}^+ \\ U_{24}^+
\end{pmatrix}(x_+,\lambda).
\end{equation}
We are ready to prove the estimates \eqref{ineq+} needed for Lemma \ref{CoercivityLemma}. Now using \eqref{DecayU1} and \eqref{DecayU2} into \eqref{EqN1byN34} yields that
\[
n_{12}^+(x_+,\lambda) = \frac{U_{21}^+ U_{13}^+ -U_{11}^+ U_{23}^+}{U_{11}^+ U_{22}^+ - U_{21}^+ U_{12}^+}(x_+,\lambda) =\frac{k+ \sigma_+(\lambda) + f_1(x_+,\lambda)}{1+ f_2(x_+,\lambda)}, 
\]
where $|f_j(x,\lambda)|= O(z_{+,\epsilon_\star}(x))$ $(j=1,2)$ uniformly in $\lambda \in [\epsilon_\star,\sqrt\frac{g}{L_0}]$ as $x\to \infty$. Hence, there exists a constant $\xi_+(\epsilon_\star)>0$ such that 
\[\begin{split}
n_{12}^+(x_+,\lambda) &\geq k+\sigma_+(\lambda) - \xi_+(\epsilon_\star)z_{+,\epsilon_\star}(x_+) \\
&\geq k+\delta(\epsilon_\star) - \xi_+(\epsilon_\star)z_{+,\epsilon_\star}(x_+).
\end{split}\]
That implies $n_{12}^+(x_+,\lambda)>0$  if 
\begin{equation}\label{EqNu+}
z_{+,\epsilon_\star}(x_+) < \frac{k+\delta(\epsilon_\star)}{\xi_+(\epsilon_\star)}.
\end{equation}
 We then estimate
\[
\Delta_+(x_+,\lambda) := (n_{11}^+(x_+,\lambda)-n_{22}^+(x_+,\lambda)- k^2-\sigma_0^2(x_+,\lambda))^2+4 n_{12}^+n_{21}^+(x_+,\lambda)
\]
Using \eqref{DecayU1} and \eqref{DecayU2} into \eqref{EqN1byN34} and \eqref{EqN2byN34} again, we have 
\[
\begin{split}
n_{11}^+(x_+,\lambda) &= k\sigma_+(\lambda) + O(z_{+,\epsilon_\star}(x_+)),\\
n_{22}^+(x_+,\lambda) +k^2+\sigma_0^2(x_+,\lambda) &=  -k\sigma_+(\lambda) + O(z_{+,\epsilon_\star}(x_+)),\\
n_{21}^+(x_+,\lambda) &=- k\sigma_+(\lambda)(k+\sigma_+(\lambda))+ O(z_{+,\epsilon_\star}(x_+)).
\end{split}
\]
Hence, there exists $w_+(\epsilon_\star)>0$ such that
\[
\begin{split}
\Delta_+(x_+,\lambda) &\leq -4k\sigma_+(\lambda)(k^2+k\sigma_+(\lambda)+\sigma_+^2(\lambda)) + w_+(\epsilon_\star)z_{+,\epsilon_\star}(x_+) \\
&\leq -4k\delta(\epsilon_\star) (k^2+k\delta(\epsilon_\star)+\delta^2(\epsilon_\star)) + w_+(\epsilon_\star)z_{+,\epsilon_\star}(x_+).
\end{split}
\]
The inequality $\Delta_+(x_+,\lambda) \leq 0$ is equivalent to
\begin{equation}\label{EqZ_+}
z_{+,\epsilon_\star}(x_+) \leq \frac{4k\delta(\epsilon_\star) (k^2+k\delta(\epsilon_\star)+\delta^2(\epsilon_\star))}{w_+(\epsilon_\star)}.
\end{equation}
Combining \eqref{EqNu+} and \eqref{EqZ_+}, we take 
\[
\Gamma_+(\epsilon_\star) = \min \Big( \frac{k+\delta(\epsilon_\star)}{\xi_+(\epsilon_\star)},  \frac{4k\delta(\epsilon_\star) (k^2+k\delta(\epsilon_\star)+\delta^2(\epsilon_\star))}{w_+(\epsilon_\star)}\Big).
\]
If $x_+$ satisfies $z_{+,\epsilon_\star}(x_+)\leq \Gamma_+(\epsilon_\star) $, then one has $n_{12}^+(x_+,\lambda) >0 \geq \Delta_+(x_+,\lambda)$, i.e. \eqref{ineq+}. That implies $BV_{x_+,\lambda}(\vartheta,\vartheta) \geq 0$.

Similarly, we get that, from \eqref{SystemN21_N22}, we follow the above arguments to show that there exists $\Gamma_-(\epsilon_\star)>0$ such that for $z_{-,\epsilon_\star}(x_-) \leq \Gamma_-(\epsilon_\star)$, \eqref{ineq-} holds. It yields  $BV_{x_-,\lambda}(\vartheta,\vartheta) \geq 0$. Proposition \ref{coercive-interval} is proven.
 \end{proof}

\section*{Acknowledgments}

The authors would like to thank Prof. Jean-Marc Delort,  Prof. Catherine Sulem and Prof. Jeffrey Rauch for fruitful discussions and Prof. Bernard Helffer and Prof. David Lannes for advice on this study. The second author thanks also to Assoc. Prof.  Quốc Anh Ngô for his encouragement. Both authors are deeply grateful to one of anonymous referees for his/her valuable comments to improve the presentation of this paper.  Parts of this work were carried out while the second author was visiting Universit\'e de Montr\'eal and he would like to thank Universit\'e de Montr\'eal for the heartfelt hospitality.  
This work is supported by a grant from R\'egion \^Ile-de-France.

\appendix 

\section{Proof of Lemma \ref{LemEigenvalueReal}}\label{Preliminaries}

Multiplying by $ \phi$ on both sides of \eqref{4thOrderEqPhi} and then integrating by parts, we  obtain that 
\begin{equation}\label{EqVariational}
\begin{split}
-\lambda^2 \int_{\R} \Big( k^2 \rho_0 |\phi|^2 + \rho_0 |\phi'|^2 \Big) dx &= \lambda \mu \int_{\R} \Big( |\phi''|^2 + 2k^2 |\phi'|^2 + k^4 |\phi|^2 \Big) dx \\
&\qquad- gk^2 \int_{\R} \rho_0'|\phi|^2 dx. 
\end{split}
\end{equation}
Suppose that $\lambda = \lambda_1 + i\lambda_2$, then one deduces from \eqref{EqVariational} that 
\begin{equation}\label{EqRealPart}
\begin{split}
-(\lambda_1^2-\lambda_2^2) \int_{\R} \Big( k^2 \rho_0 |\phi|^2 + \rho_0 |\phi'|^2 \Big) dx &= \lambda_1 \mu \int_{\R} \Big( |\phi''|^2 + 2k^2 |\phi'|^2 + k^4 |\phi|^2 \Big) dx\\
&\qquad - gk^2 \int_{\R} \rho_0'|\phi|^2 dx
\end{split}
\end{equation}
and that 
\begin{equation}\label{EqImaginaryPart}
-2\lambda_1 \lambda_2 \int_{\R} \Big( k^2 \rho_0 |\phi|^2 + \rho_0 |\phi'|^2 \Big) dx = \lambda_2 \mu \int_{\R} \Big( |\phi''|^2 + 2k^2 |\phi'|^2 + k^4 |\phi|^2 \Big) dx.
\end{equation}
If $\lambda_2 \neq 0$, \eqref{EqImaginaryPart} leads us to
\[
-2\lambda_1 \int_{\R} \Big( k^2 \rho_0 |\phi|^2 + \rho_0 |\phi'|^2 \Big) dx =  \mu \int_{\R} \Big( |\phi''|^2 + 2k^2 |\phi'|^2 + k^4 |\phi|^2 \Big) dx,
\]
which yields 
\[
\begin{split}
-(\lambda_1^2-\lambda_2^2) \int_{\R} \Big( k^2 \rho_0 |\phi|^2 + \rho_0 |\phi'|^2 \Big) dx  &= -2\lambda_1^2 \int_{\R} \Big( k^2 \rho_0 |\phi|^2 + \rho_0 |\phi'|^2 \Big) dx \\
&\qquad - gk^2 \int_{\R} \rho_0'|\phi|^2 dx.
\end{split}
\]
Equivalently, 
\begin{equation}\label{EqModuloLambda}
(\lambda_1^2+\lambda_2^2)\int_{\R} \Big( k^2 \rho_0 |\phi|^2 + \rho_0 |\phi'|^2 \Big) dx= - gk^2 \int_{\R} \rho_0'|\phi|^2 dx.
\end{equation}
That implies 
\[
(\lambda_1^2+\lambda_2^2) k^2 \inf_{\R} \rho_0 \int_{\R^2} |\phi|^2 dx \leqslant - gk^2 \int_{\R} \rho_0'|\phi|^2 dx.
\]
The positivity of $\rho_0'$ yields a contradiction, then  $\lambda$ is real. Using \eqref{EqVariational} again, we further get that 
\[
\lambda^2 \int_{\R}\rho_0(k^2|\phi|^2+|\phi'|^2) dx \leqslant gk^2 \int_{\R}\rho_0'|\phi|^2 dx.
\]
It tells us that $\lambda$ is  bounded by  $\sqrt{\frac{g}{L_0}}$. This finishes the proof of Lemma \ref{LemEigenvalueReal}.

\section{Differentiability of eigenvalues of self-adjoint and compact operators}\label{AppContinuous}

The classical perturbation theory from \cite[Chapter VII, $\$3$]{Kato} has shown the continuous property of the eigenvalues for a family of  holomorphic self-adjoint operators in an infinite-dimensional Hilbert space. If the operators are only differentiable,  we will present a proof of the differentiability  of the eigenvalues for  compact and self-adjoint operators in an infinite-dimensional Hilbert space deduced from that one for  matrix functions in a finite-dimensional space (see \cite[Chapter II, $\$5$]{Kato}).

\begin{theorem}\label{ThmEvCont}
Let $I$ be a closed interval and $H$ be an infinite-dimensional Hilbert space and $(A(\lambda))_{\lambda \in I}$ be a family of self-adjoint and compact  operators in $H$ depending continuously differentiable on $\lambda$. Then, all eigenvalues and all eigenvectors of $A(\lambda)$ are differentiable functions on $\lambda$.
\end{theorem}
\begin{proof}
Let $\lambda_0\in I$ be fixed.  Since $A(\lambda_0)$ is a self-adjoint and compact operator in $H$, the spectrum of $A(\lambda_0)$ is discrete. Let $\gamma_0$ be an arbitrary  eigenvalue of $A(\lambda_0)$ and $E= \text{Ker}(A(\lambda_0)-\gamma_0 \text{Id}_H)$, we have  the decomposition $H = E \oplus E^{\perp}$. 
Consequently, for all $\lambda \in I$,  
\[
A(\lambda)=\begin{pmatrix} 
\text{Proj}_E (A(\lambda) \text{Proj}_E)  &  \text{Proj}_E (A(\lambda) \text{Proj}_{E^{\perp}})  \\
\text{Proj}_{E^{\perp}} (A(\lambda)\text{Proj}_E)  &  \text{Proj}_{E^{\perp}}(A(\lambda) \text{Proj}_{E^{\perp}})
\end{pmatrix} 
\]
that we will denote by  $(A_{ij}(\lambda))_{1\leqslant i, j \leqslant 2}$ for brevity. 
Notice that 
\[
A(\lambda_0) = \begin{pmatrix}
\gamma_0 \text{Id}_E & 0 \\
0 & A_{22}(\lambda_0)
\end{pmatrix},
\]
and $A_{22}(\lambda_0)-\gamma_0 \text{Id}_{E^{\perp}}$ is invertible. 

Let $0< \varepsilon \ll 1$ and $\gamma$ be  an eigenvalue of $A(\lambda)$ being close to $\gamma_0$, i.e. $|\gamma-\gamma_0|<\varepsilon$.  We write that 
\begin{equation}
A(\lambda) - \gamma \text{Id} = \begin{pmatrix}
A_{11}(\lambda) - \gamma \text{Id}_E & A_{12}(\lambda) \\
A_{21}(\lambda) & A_{22}(\lambda)- \gamma \text{Id}_{E^{\perp}}
\end{pmatrix}.
\end{equation}
If $x= (y,z)^T$ is a corresponding eigenvector, we obtain
\[
\begin{cases}
A_{11}(\lambda) y + A_{12}(\lambda) z =\gamma y, \\
A_{21}(\lambda) y + A_{22}(\lambda) z =\gamma z.
\end{cases}
\]
Consequently, 
\begin{equation}\label{EqSystA12}
\begin{cases}
(A_{11}(\lambda)-\gamma \text{Id}_E) y + A_{12}(\lambda) z=0,\\
A_{21}(\lambda) y + (A_{22}(\lambda) - \gamma \text{Id}_{E^{\perp}})z = 0.
\end{cases}
\end{equation}

Since $A_{22}(\lambda_0) - \gamma_0 \text{Id}_{E^{\perp}}$ is invertible,  we have that   $A_{22}(\lambda) - \gamma_0 \text{Id}_{E^{\perp}}$ is  invertible  for $|\lambda - \lambda_0|<\delta \ll 1$. We further get that $A_{22}(\lambda)-\gamma \text{Id}_{E^{\perp}}$ is also invertible for $|\lambda-\lambda_0| < \delta$ and $|\gamma-\gamma_0|<\varepsilon$. 
Hence, we deduce that \eqref{EqSystA12} is equivalent to 
\begin{equation}\label{EqYandZ}
\begin{cases}
z=- (A_{22}(\lambda)-\gamma \text{Id}_{E^{\perp}})^{-1} A_{21}(\lambda) y,\\
(A_{11}(\lambda) - A_{12}(\lambda) (A_{22}(\lambda)-\gamma \text{Id}_{E^{\perp}})^{-1} A_{21}(\lambda)) y= \gamma y.
\end{cases}
\end{equation}
If $y=0$, $\eqref{EqYandZ}_1$ implies $z=0$, which is impossible.  We have $y\neq 0$, this means that if $\gamma \in (\gamma_0-\varepsilon, \gamma_0+\varepsilon)$ is an eigenvalue of $A(\lambda)$, $\eqref{EqYandZ}_2$ tells us that $\gamma$ is also an eigenvalue of $B(\lambda, \gamma)$ defined by 
\[
B(\lambda,\gamma) := A_{11}(\lambda) - A_{12}(\lambda) (A_{22}(\lambda)-\gamma \text{Id}_{E^{\perp}})^{-1} A_{21}(\lambda) : E \to E
\]
Notice that $E$ is finite-dimensional thanks to Riesz's theorem. $B(\lambda,\gamma)$ turns out to be a matrix, having eigenvalues $\gamma_j(\lambda,\gamma) (1\leqslant j\leqslant \text{dim}E)$. Then, there exists $j$ such that $\gamma_j(\lambda,\gamma)=\gamma$. It follows from \cite[Chapter II, $\$5$]{Kato} that $\gamma_j(\lambda,\gamma)$ and its associated eigenvector are differentiable at $\lambda_0$, so is $\gamma$. 
\end{proof}

\section{Decaying solutions of the linear system of ODEs}\label{ThmGeneralSystem}
We state here the key ingredient for our analysis in Section \ref{SectGeneralProf} due to E. A. Coddington and N. Levinson \cite[Theorem 8.1, Chapter 3]{CL81}.

\begin{theorem}\label{Theorem81}
We consider a linear system 
\begin{equation}\label{EqGeneral}
W'(y)=(A+L(y)+R(y))W(y).
\end{equation}
Let $A$ be a constant matrix with characteristic roots $\mu_j, j=1,\dots,n$, all of which are distinct. Let the matrix $L$ be differentiable and satisfy 
\begin{equation}\label{ConditionOnL}
\int_0^{\infty} \|L'(y)\|dy <\infty
\end{equation}
and let $L(y)\to 0 $ as $y\to \infty$. Let the matrix $R$ be integrable and let 
\begin{equation}\label{ConditionOnR}
\int_0^{\infty} \|R(y)\|dy <\infty.
\end{equation}
Let the roots of $\text{det}(A+L(y)-\lambda I_n)=0$ be denoted by $\lambda_j(y), j=1,\dots,n$. Clearly, by reordering the $\mu_j$ if necessary, $\lim_{y\to \infty} \lambda_j(y) =\mu_j$. For a given $h$, let 
\[
d_{hj}(y) = \text{Re}(\lambda_h(y)-\lambda_j(y)).
\]
Suppose all $\lambda_j  (1\leqslant j \leqslant n)$  fall into one of two classes $H_1$ and $H_2$, where
\[
\lambda_j\in H_1 \quad \text{if } \int_0^y d_{hj}(s) ds \to \infty \text{ as } y\to \infty \text{ and } \int_{y_1}^{y_2} d_{hj}(s) ds \geqslant -K \quad(y_2 \geqslant y_1 \geqslant 0),
\]
and
\[
\lambda_j\in H_2 \quad \text{if } \int_{y_1}^{y_2} d_{hj}(s) ds \leqslant K\quad(y_2 \geqslant y_1 \geqslant 0),
\]
where $h$ is fixed and $K$ is a constant. Let $p_h $ be the eigenvector corresponding to $\mu_h$, i.e. $Ap_h=\mu_h p_h$. Hence, there is a solution $\phi_h$ of \eqref{EqGeneral} and a $y_0 \in (0,\infty)$ such that 
\[
\lim_{y\to \infty} \phi_h(y) \text{exp}\Big[ -\int_{y_0}^y \lambda_h(s) ds \Big] = p_h.
\]
\end{theorem}

\section{A remark on the relation between the formulation on $[x_-, x_+]$ and the formulation on $\R$ of the viscous  RT problem}

\begin{proposition}\label{PropD1}
For all $\phi$ bounded solution of \eqref{4thOrderEqPhi} on $\R$ and for all $\theta \in H^2(\R)$, there holds
\begin{equation}\label{BilinearEquivalence}
\begin{split}
&\lambda  \int_{-\infty}^{+\infty}   \rho_0 (k^2\phi  \theta + \phi'  \theta') dx + \mu \int_{-\infty}^{+\infty} (\phi''  \theta'' + 2k^2 \phi  \theta' +k^4 \phi  \theta)dx \\
&\qquad= \bB_{x_-,x_+,\lambda}(\phi,\theta) + \int_{\R\setminus[x_-,x_+]} \frac{gk^2\rho_0'}{\lambda} \phi\theta dx.
\end{split}
\end{equation}
\end{proposition}
We immediately have two remarks from \eqref{BilinearEquivalence}. 
\begin{enumerate}
\item
In the case of $\rho_0'$ being compactly supported ($\text{supp}\rho_0'=[-a,a]$), we have 
\begin{equation}
\lambda  \int_{-\infty}^{+\infty}   \rho_0 (k^2\phi  \theta + \phi'  \theta') dx + \mu \int_{-\infty}^{+\infty} (\phi''  \theta'' + 2k^2 \phi  \theta' +k^4 \phi  \theta)dx = \bB_{a,\lambda}(\phi,\theta).
\end{equation}
That means $\bB_{x_-,x_+,\lambda}$ is independent of $x_\pm$ if and only if $x_-\leq -a <a \leq x_+$.

\item In the case $\rho_0'>0$ everywhere, for each $(x_-,x_+)$, a penalization of $\bB_{x_-,x_+,\lambda}$ by the term $\int_{x_-}^{x_+}\frac{gk^2\rho_0'}{\lambda} \phi\theta dx$ is necessary to obtain the ODE \eqref{4thOrderEqPhi} on the whole space. 
\end{enumerate}

\begin{proof}[Proof of Proposition \ref{PropD1}]
To prove \eqref{BilinearEquivalence}, we show two following identities 
\begin{equation}\label{BVatX_+}
\begin{split}
&\int_{x_+}^{+\infty} \lambda\rho_0(k^2\phi \theta+\phi' \theta')dx+\mu\int_{x_+}^{+\infty}(k^4\phi \theta+2k^2\phi' \theta'+ \phi'' \theta'')dx \\
&\qquad= \int_{x_+}^{+\infty} \frac{gk^2\rho_0'}{\lambda}\phi \theta dx  + BV_{x_+,\lambda}(\phi,\theta).
\end{split}
\end{equation}
and 
\begin{equation}\label{BVatX_-} 
\begin{split}
&\int_{-\infty}^{x_-} \lambda\rho_0(k^2\phi \theta+\phi' \theta')dx+\mu\int_{-\infty}^{x_-} (k^4\phi \theta+2k^2\phi' \theta'+ \phi'' \theta'')dx \\
&\qquad=\int_{-\infty}^{x_-} \frac{gk^2\rho_0'}{\lambda}\phi \theta dx  + BV_{x_+,\lambda}(\phi,\theta).
\end{split}
\end{equation}

We have the following remark that we state on $(x_+, +\infty)$ (a similar expression holds true for $(-\infty,x_-)$) that if $\phi$ is a bounded solution of \eqref{4thOrderEqPhi} on $(x_+,+\infty)$, we then have from Proposition \ref{PropEquivalentEq} that $\phi$ satisfies \eqref{GenRightBound} at $x_+$.
We integrate by parts to have that 
\begin{equation}\label{1stEqAppD}
\begin{split}
\int_{x_+}^{+\infty} \phi'' \theta'' dx &= \phi'' \theta'\Big|_{x_+}^{+\infty} - \phi'''\theta\Big|_{x_+}^{+\infty}+ \int_{x_+}^{+\infty} \phi^{(4)}\theta dx \\
&= -(\phi''\theta')(x_+) +(\phi'''\theta)(x_+) +  \int_{x_+}^{+\infty} \phi^{(4)}\theta dx.
\end{split}
\end{equation}
Because of \eqref{GenRightBound}, we have
\begin{equation}\label{2ndEqAppD}
\phi''(x_+) = -(n_{11}^+\phi(x_+) + n_{12}^+\phi'(x_+)) 
\end{equation}
and 
\begin{equation}\label{3rdEqAppD}
\phi'''(x_+) = -(n_{21}^+\phi(x_+) + n_{22}^+\phi'(x_+)). 
\end{equation}
Substituting \eqref{2ndEqAppD} and \eqref{3rdEqAppD} into \eqref{1stEqAppD}, we obtain 
\begin{equation}\label{4thEqAppD}
\begin{split}
\int_{x_+}^{+\infty} \phi'' \theta'' dx &=(n_{11}^+\phi(x_+) + n_{12}^+\phi'(x_+)) \theta'(x_+) -(n_{21}^+\phi(x_+) + n_{22}^+\phi'(x_+)) \theta(x_+)\\
&\qquad\quad+  \int_{x_+}^{+\infty} \phi^{(4)}\theta dx.
\end{split}
\end{equation}
Using the integration by parts again, we deduce that 
\begin{equation}\label{5thEqAppD}
\begin{split}
\int_{x_+}^{+\infty} \phi' \theta' dx &= \phi' \theta\Big|_{x_+}^{+\infty} -\int_{x_+}^{+\infty} \phi'' \theta dx = -(\phi' \theta)(x_+) - \int_{x_+}^{+\infty} \phi'' \theta dx 
\end{split}
\end{equation}
and that 
\begin{equation}\label{6thEqAppD}
\begin{split}
\int_{x_+}^{+\infty} \rho_0\phi' \theta'dx &= - (\rho_0\phi' \theta)(x_+) - \int_{x_+}^{+\infty} (\rho_0\phi')' \theta dx.
\end{split}
\end{equation}
In view of \eqref{4thEqAppD}, \eqref{5thEqAppD} and \eqref{6thEqAppD}, we obtain 
\begin{equation}\label{7thEqAppD}
\begin{split}
&\int_{x_+}^{+\infty} \lambda\rho_0(k^2\phi \theta+\phi' \theta')dx+\mu\int_{x_+}^{+\infty}(k^4\phi \theta+2k^2\phi' \theta'+ \phi'' \theta'')dx \\
&= \int_{x_+}^{+\infty} (\lambda(k^2\rho_0\phi- (\rho_0\phi')')+ \mu(\phi^{(4)}-2k^2\phi''+k^4\phi))\theta dx  \\
&\qquad+ \mu (n_{11}^+\phi(x_+) + n_{12}^+\phi'(x_+)) \theta'(x_+) -\mu (n_{21}^+\phi(x_+) + n_{22}^+\phi'(x_+)) \theta(x_+) \\
&\qquad -2 k^2 \mu (\phi' \theta)(x_+) - \lambda (\rho_0\phi' \theta)(x_+)\\
&= \int_{x_+}^{+\infty} \frac{gk^2\rho_0'}{\lambda}\phi \theta dx  + BV_{x_+,\lambda}(\phi,\theta).
\end{split}
\end{equation}
Hence, \eqref{BVatX_+} follows from \eqref{7thEqAppD}.  Similarly, we obtain \eqref{BVatX_-}.
It follows from \eqref{BVatX_+} and \eqref{BVatX_-}  that 
\begin{equation}
\begin{split}
&\lambda  \int_{-\infty}^{+\infty}   \rho_0 (k^2\phi  \theta + \phi'  \theta') dx + \mu \int_{-\infty}^{+\infty} (\phi''  \theta'' + 2k^2 \phi  \theta' +k^4 \phi  \theta)dx \\
&= \int_{-\infty}^{x_-} \frac{gk^2\rho_0'}{\lambda}\phi \theta dx  + \int_{x_+}^{+\infty} \frac{gk^2\rho_0'}{\lambda}\phi \theta dx + BV_{x_-,\lambda}(\phi,\theta) +BV_{x_+,\lambda}(\phi,\theta) \\
&\qquad + \int_{x_-}^{x_+} \lambda\rho_0(k^2\phi \theta+\phi' \theta')dx+\mu\int_{x_-}^{x_+} (k^4\phi \theta+2k^2\phi' \theta'+ \phi'' \theta'')dx \\
&=  \int_{-\infty}^{x_-} \frac{gk^2\rho_0'}{\lambda}\phi \theta dx  + \int_{x_+}^{+\infty} \frac{gk^2\rho_0'}{\lambda}\phi \theta dx + \bB_{x_-,x_+,\lambda}(\phi,\theta).
\end{split}
\end{equation}
It yields \eqref{BilinearEquivalence}. 
\end{proof}

%\newpage

%\begin{center}
%\Large\textbf{Answers to referees}
%\end{center}

%\section*{Answer to Referee $\#$1}

%Yes, one can choose $x_-=-x_+$ to obtain a symmetric interval  as in the section with $\rho_0'$ of compact support. But we still use the notations $x_\pm$ in Lemma 4.1 and Proposition 4.2 to  obtain a criterion for $x_-$ and $x_+$ depending on the precise behavior of $\rho_0$ in Section 4.4. Such estimates can be used to derive results uniform  in $k$, $\lambda$ and to obtain better estimates in the high frequency regime, that's what we expect to prove $N(\epsilon_\star)\to \infty$ as $\epsilon_\star\to 0$.

\end{document}